%% file: arXiv_v2.tex
\documentclass{article}
\input{Preamble}
\usepackage{enumitem}
\usepackage{mathdots}
\setlist[enumerate]{noitemsep}
\setlist[itemize]{leftmargin=*,noitemsep}
\newcommand{\GGL}{{\overline{\mathcal{GL}}}}

\title{\textbf{Almost abelian pseudo-Kähler Lie algebras}}

\author{Diego Conti \orcidlink{0000-0001-6812-4411} and Alejandro Gil-García \orcidlink{0000-0002-9370-241X}}

\date{\today}

\begin{document}

\maketitle

\begin{abstract}

We study invariant pseudo-Kähler structures on a solvmanifold $G$ such that the Lie algebra $\mathfrak{g}$ is almost abelian, that is $\mathfrak{g}=\mathfrak{h}\rtimes\mathbb{R}$, with $\mathfrak{h}$ abelian; comparing with the positive-definite case, an additional situation occurs, corresponding to the ideal $\mathfrak{h}$ being degenerate. We obtain a classification up to unitary isomorphism in all dimensions. We deduce that every nilpotent almost abelian Lie algebra endowed with a complex structure also admits a compatible pseudo-Kähler structure, and prove that this is no longer true for general almost abelian Lie algebras; indeed, we classify all the almost abelian Lie algebras that admit a complex structure and a symplectic structure but no compatible pseudo-Kähler metric. We study the curvature of the metrics we have obtained, and use some of them to construct Einstein pseudo-Kähler metrics in two dimensions higher.\bigskip

\medskip


\end{abstract}

\clearpage

\tableofcontents


\section*{Introduction}

A recurrent theme in Riemannian geometry is the quest for metrics satisfying both curvature conditions and compatibility with an additional geometric structure, the prototype being Kähler-Einstein geometry. Among the possible holonomy groups of an irreducible non-symmetric metric, barring the case of the full special orthogonal group, $\mathrm{U}(n)$ is the group that puts less constraints on the metric; however, Riemannian manifolds with holonomy contained in $\mathrm{U}(n)$, i.e.\ Kähler, enjoy several properties which make their study more amenable. The indefinite picture is more intricate, due to the existence of irreducible holonomy groups that act decomposably; nonetheless, the holonomy group $\mathrm{U}(p,q)$ shares some properties with its definite counterpart. For instance, in \cite{Bolsinov_Rosemann_2021}, a local characterization of Bochner-flat manifolds was derived, which holds in both the Kähler and pseudo-Kähler setting; in \cite{Derdzinski_Terek_2024}, a generalization of the classical result that Killing vector fields on a compact Kähler manifold are holomorphic was obtained, and a relation with Betti numbers was found in \cite{Yamada_2019}. Extending to the indefinite setting the classification results of Kähler geometry and understanding whether the properties of Kähler manifolds generalize to arbitrary signature has also been a subject of enquiry (see e.g.\ \cite{Petean97,LU21}). Regardless of the extent to which Kähler and pseudo-Kähler metrics can be likened, it is a natural question, in the presence of any construction of ``special'' pseudo-Riemannian metrics, whether it is possible to choose the metric to be pseudo-Kähler.\medskip

Since pseudo-Kähler metrics are defined by a symplectic form and a complex structure, one can arrive to the same problem from a different perspective: given a class of complex or symplectic manifolds, for which of them is it possible to choose a complex and a symplectic structure in a compatible way, so that a metric is defined? From this point of view, there is no particular reason to assume that the metric be positive-definite.\medskip

The pseudo-Kähler structures that we consider in this paper are homogeneous; more precisely, left-invariant on Lie groups. The positive-definite case is well understood: the structure of homogeneous Kähler manifolds was determined in \cite{DN88}; for left-invariant metrics, Lie algebras supporting a Kähler metric are characterized in \cite{Lic88,DM96}. Notably, the nilpotent case is completely trivial: it was shown in \cite{Hasegawa} that the only Kähler nilmanifolds are tori. Kähler-Einstein homogeneous manifolds are also quite scarce; they fall into three classes: flag manifolds in positive curvature (see \cite{Matsushima_1972}), flat manifolds classified by \cite{Milnor_1976}, and standard Einstein solvmanifolds (see \cite{Lauret_2010,Bohm_Lafuente_2023}), whose structure was described in \cite{Heber_1998}.\medskip

The indefinite case is much less restrictive. Compact homogeneous pseudo-Kähler manifolds are classified in \cite{DG92} (with a generalization to non-invariant metrics in \cite{Guan10}), but little seems to be known about pseudo-Kähler Lie algebras in general. Classifications do exist, but only in low dimensions and with assumptions (\cite{Ova06,CFU04,LU25,CG25}). Restricted classes have also been studied: those with an abelian complex structure (\cite{BS12}), special Kähler, relevant in the construction of quaternion-Kähler metrics (\cite{Valencia,Man21,Ac00}), hypersymplectic (\cite{AD06,And06,BGL21}) and Born (\cite{BornGeometry25,BornLieAlgebras25}). Examples of Einstein pseudo-Kähler solvmanifolds have been constructed in \cite{Yamada12,Conti_Rossi_SegnanDalmasso_2023,Rossi:NewSpecialFantastic}.\medskip

The lack of general results on pseudo-Kähler Lie algebras may be attributed to the fact that general Lie algebras are too flexible; for instance, nilpotent or solvable Lie algebras are only classified in low dimensions. General results are easier to obtain for the restricted class of almost abelian Lie algebras, which are completely determined by the conjugacy class of an endomorphism. Some very general results for structures on almost abelian Lie algebras have been proved in \cite{Freibert24_TFHS}, and various non-Kähler geometries related to complex geometry have been considered in the literature (\cite{Paradiso21,BFLT23,FP21,FP23}). For our purposes, the key result in this setting is the classification of the almost abelian Lie algebras admitting complex or symplectic structures in \cite{ArroyoBarberisDiazGodoyHernandez}.\medskip

In this paper, we study almost abelian pseudo-Kähler Lie algebras. A Lie algebra $\frg$ is \emph{almost abelian} if it admits a codimension one abelian ideal $\frh$, thus $\frg$ can be identified with $\frh\rtimes_D\R$, where $D\in\End(\frh)$. Throughout the paper we will need to distinguish two different situations, depending whether the abelian ideal $\frh$ is degenerate or not with respect to the pseudo-Kähler metric (equivalently, if the one-dimensional factor $\R$ is isotropic or not). We refer to the former as the \emph{isotropic case} and to the latter as the \emph{non-isotropic case}. We describe these two situations in Section \ref{section:characterization}, where we give characterizations of the endomorphism $D$ for the non-isotropic case in Proposition~\ref{prop:characterization_non-isotropic_case} and for the isotropic case in Proposition~\ref{prop:characterization_isotropic_case}, in accord with \cite{Freibert24_TFHS}. These characterizations allow us to conclude that the first Betti number of a nilpotent almost abelian pseudo-Kähler Lie algebra is always greater than two (see Proposition \ref{prop:b1>2_in_PK}). In Section \ref{section:curvature} we determine the curvature of an almost abelian pseudo-Kähler Lie algebra. We show that the pseudo-Kähler metric is either an algebraic Ricci soliton (see Corollary \ref{cor:Ricci_soliton}) or Ricci-flat (see Proposition \ref{prop:curvature_isotropic_h0}), and flat if the underlying Lie algebra is unimodular (see Corollary \ref{cor:flat_almost_ab_pK}).\medskip

Building on the classification of adjoint orbits of $\mathrm{U}(p,q)$ obtained in \cite{Burgoyne_Cushman_1977} (see Theorem \ref{thm:orbits_U(p,q)}), we give a complete classification of almost abelian pseudo-Kähler Lie algebras in Section \ref{section:classification}. To this end, we introduce the language of \emph{blocks} and the abelian semigroup $\mathcal{U}$ freely generated by them. A block $\Delta$ consists of a complex vector space $V$, a Hermitian form $h$ on $V$, and a skew-Hermitian endomorphism $A\in\End(V)$ relative to $h$. We introduce seven families of Lie algebras $\frg_0,\ldots,\frg_6$ depending on types $t\in\mathcal{U}$ and each of them equipped with a fixed pseudo-Kähler structure. We then show in Theorem \ref{thm:classification} that every almost abelian pseudo-Kähler Lie algebra $\frg$ belongs to one of these families. We carry out our classification in arbitrary dimension, then specialize to dimensions six and eight (Subsections \ref{section:6d_classification} and \ref{section:8d_classification}) in more explicit terms; we believe that this explicit description will find applications beyond the scope of this paper.\medskip

After giving this classification up to isomorphisms preserving metric and complex structure, we consider the related question of which isomorphism classes of almost abelian Lie algebras admit a pseudo-Kähler structure. We give a complete answer in Theorem \ref{thm:jordantypes}. In the special case of nilpotent almost abelian Lie algebras, it follows from \cite{ArroyoBarberisDiazGodoyHernandez} that the existence of a complex structure implies the existence of a symplectic one; moreover, the complex structure is unique by \cite{AABRW25}. Comparing with our classification, we show in Corollary \ref{cor:complex_PK} that every nilpotent almost abelian Lie algebra with a complex structure admits a compatible pseudo-Kähler metric. Outside the nilpotent case, we prove that there exist almost abelian Lie algebras that admit both a complex and a symplectic structure, but not a pseudo-Kähler structure, and classify them in Theorem \ref{thm:notpk}.\medskip

Finally, we show that some of the Ricci-flat pseudo-Kähler metrics obtained with our method can be used to produce pseudo-Kähler-Einstein Lie algebras in two dimensions higher with the construction of \cite{Conti_Rossi_SegnanDalmasso_2023}. In Section \ref{section:PKE_extensions}, we write down the conditions that a nilpotent almost abelian pseudo-Kähler Lie algebra must satisfy in order for the construction to apply (see Propositions \ref{prop:ESnil_non-isotropic} and \ref{prop:ESnil}). We conclude with an explicit classification of the eight-dimensional pseudo-Kähler-Einstein Lie algebras that can be obtained in this way (Propositions \ref{prop:KEextension6d_non_isotropic} and \ref{prop:KEextension6d}).

\subsection*{Acknowledgements}

The authors are grateful to Marco Freibert and Federico A.\ Rossi for useful comments and to the referees for their valuable comments. D.\ C.\ would like to acknowledge the project PRIN 2022MWPMAB ``Interactions between Geometric Structures and Function Theories'', GNSAGA of INdAM and  the MIUR Excellence Department Project awarded to the Department of Mathematics, University of Pisa, CUP I57G22000700001. The work of A.\ G.\ is supported by the Scuola Internazionale Superiore di Studi Avanzati (SISSA) and was previously funded by the Beijing Institute of Mathematical Sciences and Applications (BIMSA).


\section{Characterization of almost abelian pseudo-Kähler Lie algebras}\label{section:characterization}

Let $\g$ be an \emph{almost abelian Lie algebra}, namely admitting a codimension one abelian ideal $\frh$; suppose $\g$ has real dimension $2n$. In particular, the derived subalgebra $[\frg,\frg]$ is abelian, and $\frg$ is two-step solvable. We can choose a basis $\{e_1,\ldots,e_{2n}\}$ of $\frg$ such that \begin{equation}\label{eqn:amostabelianbasis}
    \frh=\Span{e_1,\ldots,e_{2n-1}},\quad\ad(e_{2n})\lie h\subset\lie h.
\end{equation}

The whole Lie algebra structure of $\frg$ is determined by the derivation $$D:=\ad(e_{2n})|_{\frh}\in\Der(\frh)=\End(\R^{2n-1}),$$ allowing one to identify $\frg$ with the semidirect product $\frh\rtimes_D\R$.\medskip

An \emph{almost pseudo-Hermitian structure} on a vector space $V$ is a pair $(J,g)$, with $J\colon V\to V$ an almost complex structure and $g\colon V\times V\to\R$ a non-degenerate scalar product such that $$g(JX,JY)=g(X,Y).$$

Given vector spaces $V$, $V'$ endowed with almost pseudo-Hermitian structures $(J,g)$, $(J',g')$, we will say that a linear isomorphism $\phi\colon V\to V'$ is \emph{unitary} if $\phi\circ J=J'\circ\phi$ and $\phi^*g'=g$. We will denote by $\mathrm{U}(V,J,g)$ the Lie group of unitary transformations of $V$, and by $\lie{u}(V,J,g)$ its Lie algebra. This definition is mostly useful when $V$ is a Lie algebra, whose corresponding Lie group has an induced left-invariant almost pseudo-Hermitian structure, but we will occasionally need to consider the case where $V$ is a subspace in the Lie algebra of interest.\medskip

Let $\frg$ be an almost abelian Lie algebra equipped with an almost pseudo-Hermitian structure $(J,g)$. By the equality $\dim\frh+\dim\frh^\perp=\dim\frg$, we see that $\frh^\perp$ is one-dimensional; consider the following two-dimensional space: $$\lie h_0:=\lie h^\perp\oplus J\lie h^\perp.$$

Since $\frh_0$ is $J$-invariant, the metric restricted to $\lie h_0$ is either definite or zero. In the first case we will say that $\frh_0$ is \emph{non-isotropic}, while in the second case we will say that $\frh_0$ is \emph{isotropic}.\medskip

Our goal is to characterize the endomorphism $D\in\End(\R^{2n-1})$ so that the almost pseudo-Hermitian structure $(J,g)$ becomes pseudo-Kähler. For this we need to distinguish two cases according to whether $\lie{h}_0$ is isotropic or not.\medskip

In the non-isotropic case, that is when the metric $g$ restricted to $\frh_0$ is definite, we have a $J$-invariant and orthogonal decomposition $$\g=\h_1\oplus \h_0.$$

Therefore, we may choose a basis $\{e_1,\dotsc, e_{2n}\}$ of $\frg$, with dual basis $\{e^1,\dotsc, e^{2n}\}$, satisfying \eqref{eqn:amostabelianbasis} and in addition \begin{equation}\label{eq:non-isotropic_pseudo-Hermitian}
    Je_{2i-1}=e_{2i}, \quad g=\sum_{i=1}^{2n}\varepsilon_i e^i\otimes e^i, \quad \lie h_1=\Span{e_1,\dotsc, e_{2n-2}}, \quad \lie h_0=\Span{e_{2n-1},e_{2n}},
\end{equation} where $\varepsilon_i:=g(e_i,e_i)\in\{-1,+1\}$. Thus, $J\frh^\perp=\Span{e_{2n-1}}$ and $\frh^\perp=\Span{e_{2n}}$. Denote $J_1:=J|_{\frh_1}$ and $g_1:=g|_{\frh_1\times\frh_1}$.

\begin{proposition}[\protect{\cite{Freibert24_TFHS}}]\label{prop:characterization_non-isotropic_case}
    Let $\frg=\frh\rtimes_D\R$ be an almost abelian Lie algebra equipped with an almost pseudo-Hermitian structure $(J,g)$ such that $\lie h_0$ is non-isotropic. Then $(J,g)$ defines a pseudo-Kähler structure on $\frg$ if and only if, with respect to the basis $\{e_1,\dotsc, e_{2n}\}$ satisfying \eqref{eqn:amostabelianbasis} and \eqref{eq:non-isotropic_pseudo-Hermitian},
    \begin{equation}\label{eq:D_positive_g0}
        D=\begin{pmatrix}
            A&0\\
            0&a
        \end{pmatrix}
    \end{equation} with $A\in\mathfrak{u}(\frh_1,J_1,g_1)$ and $a\in\R$.
\end{proposition}

\begin{proof}
    This is proved in \cite[Example 2.45]{Freibert24_TFHS}, but also follows from the same argument given in the positive-definite Kähler case in \cite{FP21}, which we reproduce here for completeness. With respect to the basis $\{e_1,\dotsc, e_{2n}\}$ of $\frg$, the derivation $D\in\End(\R^{2n-1})$ takes the form $$D=\begin{pmatrix}
        A&v\\
        w^t&a
    \end{pmatrix},$$ where $A\in\mathfrak{gl}(\frh_1)=\End(\R^{2n-2})$, $v,w\in\frh_1\cong\R^{2n-2}$ and $a\in\R$. The complex structure $J$ is integrable if and only if $D$ commutes with $J_1$, which amounts to $[A,J_1]=0$, and $\frh_1$ is invariant under the action of $D$, i.e.\ $w=0$ (see \cite{LRv15}). Arguing as in \cite{FP21}, we conclude that the pseudo-Hermitian pair $(J,g)$ is pseudo-Kähler if and only if, in addition, $v=0$ and $A$ is skew-symmetric with respect to $g_1$. Thus $D$ is as in the statement.
\end{proof}

In the isotropic case, that is when the metric $g$ restricted to $\frh_0$ is zero, the metric on $\lie h$ is degenerate. This case can too be deduced from the general result of \cite[Theorem 2.43]{Freibert24_TFHS}, but we will follow a more direct approach.\medskip

Fix a generator $e_1$ of $\lie h^\perp$; then $e_1$ belongs to $\lie h$. We will complete $e_1$ to a basis adapted to the filtration $$\lie h^\perp\subset \lie h_0\subset \lie h_0^\perp\subset\frh\subset\lie g.$$

Choose a $J$-invariant complement $V$ of $\lie h_0$ in $\lie h_0^\perp$. Since $\lie h_0$ is precisely the null space of $\lie h_0^\perp$, the complement $V$ is non-degenerate, and has an almost pseudo-Hermitian structure $(J_V,g_V)$ induced by restriction. Choose a basis $\{e_3,\dotsc, e_{2n-2}\}$ of $V$ such that $$J_Ve_{2i-1}=e_{2i}, \quad g_V=\sum_{i=3}^{2n-2}\varepsilon_i e^i\otimes e^i.$$

We can also choose a vector $e_{2n}$ orthogonal to $V$ such that $$g(e_1,e_{2n})=1,\quad g(Je_1,e_{2n})=0,\quad g(e_{2n},e_{2n})=0.$$

Necessarily, $e_{2n}$ is not in $\lie h$. Defining $e_2=Je_1$ and $e_{2n-1}=-Je_{2n}$, we obtain a basis $\{e_1,\dotsc, e_{2n}\}$ of $\frg$ such that \begin{equation}\label{eqn:almosthermitianstructurezero}
    Je_{2i-1}=e_{2i}, \quad g=e^1\odot e^{2n}-e^2\odot e^{2n-1}+\sum_{i=3}^{2n-2}\varepsilon_i e^i\otimes e^i, \quad \lie h=\Span{e_1,\dotsc, e_{2n-1}},
\end{equation} where $e^i\odot e^j:=e^i\otimes e^j+e^j\otimes e^i$. In this case we consider the decomposition $$\frg=\frh^\perp\oplus J\frh^\perp\oplus V\oplus\Span{e_{2n-1}}\oplus\Span{e_{2n}},$$ where $$\frh^\perp=\Span{e_1},\quad J\frh^\perp=\Span{e_2},\quad V=\Span{e_3,\ldots,e_{2n-2}}.$$

Denote by $v^{\flat_V}$ the musical isomorphism $V\to V^*$ induced by the metric $g_V$.

\begin{proposition}\label{prop:characterization_isotropic_case}
    Let $\frg=\frh\rtimes_D\R$ be an almost abelian Lie algebra equipped with an almost pseudo-Hermitian structure $(J,g)$ such that $\lie h_0$ is isotropic. Then $(J,g)$ defines a pseudo-Kähler structure on $\frg$ if and only if, with respect to the basis $\{e_1,\dotsc, e_{2n}\}$ satisfying \eqref{eqn:amostabelianbasis} and \eqref{eqn:almosthermitianstructurezero},
    \begin{equation}\label{eq:D_zero_g0}
        D=\begin{pmatrix}
            -a&0&-(J_Vv)^{\flat_V}&c_1\\
            0&-a&v^{\flat_V}&c_2\\
            0&0&A&v\\
            0&0&0&a
        \end{pmatrix}
    \end{equation} with $A\in\mathfrak{u}(V,J_V,g_V)$, $v\in V$ and $a,c_1,c_2\in\R$.
\end{proposition}

\begin{proof}
    With respect to the decomposition $\frh=\frh^\perp\oplus J\frh^\perp\oplus V\oplus\Span{e_{2n-1}}$, the derivation $D$ takes the form $$D=\begin{pmatrix}
        h_1&h_2&w_1^t&c_1\\
        h_3&h_4&w_2^t&c_2\\
        v_1&v_2&A&v_3\\
        c_3&c_4&w_3^t&a
    \end{pmatrix},$$ where $A\in\mathfrak{gl}(V)=\End(\R^{2n-4})$, $v_1,v_2,v_3,w_1,w_2,w_3\in V\cong\R^{2n-4}$ and $h_1,h_2,h_3,h_4,c_1,c_2,c_3,c_4,a\in\R$.\medskip

    Let us first obtain the conditions on $D$ for $J$ to be integrable, that is the conditions on $D$ for when the Nijenhuis tensor $N_J$ vanishes. Denote $W:=\frh^\perp\oplus J\frh^\perp\oplus V\subset\frh$. Clearly $N_J(X,Y)=0$ for $X,Y\in W$. Let $X\in W$, then \begin{align*}
        N_J(X,e_{2n-1})&=[X,e_{2n-1}]+J[JX,e_{2n-1}]+J[X,Je_{2n-1}]-[JX,Je_{2n-1}]\\
        &=J[X,e_{2n}]-[JX,e_{2n}]=-JDX+DJX,
    \end{align*} which is zero if and only if $D$ preserves the subspace $W$, i.e.\ $c_3=c_4=0$ and $w_3=0$, and $D$ commutes with $J|_W$. Note that $$N_J(X,e_{2n})=N_J(X,Je_{2n-1})=-JN_J(X,e_{2n-1})=0,$$ and it is also clear that $N_J(e_{2n-1},e_{2n})=0$. The condition $[D,J|_W]=0$ implies that $h_4=h_1$, $h_3=-h_2$, $v_2=J_Vv_1$, $w_2^t=-w_1^tJ_V$ and $[A,J_V]=0$.\medskip
    
    We now obtain the conditions on $D$ for $(J,g)$ to be pseudo-Kähler, that is the conditions on $D$ for the fundamental two-form $\omega:=g(J\cdot,\cdot)$ to be closed. This amounts to showing that $$\d\omega(X,Y,Z)=g([X,Y],JZ)+g([Z,X],JY)+g([Y,Z],JX)=0$$ for all $X,Y,Z\in\frg$. This is clear if $X,Y,Z\in W$, where $W=\frh^\perp\oplus J\frh^\perp\oplus V\subset\frh$. Let $X,Y\in W$. Then $\d\omega(X,Y,e_{2n-1})=0$ and $$\d\omega(X,Y,e_{2n})=g([e_{2n},X],JY)+g([Y,e_{2n}],JX)=g(DX,JY)-g(JX,DY).$$

    We distinguish the following cases: \begin{itemize}
        \item If $X,Y\in\frh^\perp\oplus J\frh^\perp$, then $\d\omega(X,Y,e_{2n})=0$.
        \item If $X=e_1\in\frh^\perp$ and $Y\in V$, then $\d\omega(e_1,Y,e_{2n})=g(v_1,JY)$, which is zero for all $Y$ if and only if $v_1=0$. The same conclusion holds if $X\in J\frh^\perp$.
        \item If $X,Y\in V$, then $$\d\omega(X,Y,e_{2n})=g(AX,JY)-g(AY,JX)=g((A+A^*)X,JY),$$ where $A^*$ denotes the adjoint of $A$ with respect to $g_V$, and we have used $g(J\cdot,\cdot)=-g(\cdot,J\cdot)$ and $[A,J_V]=0$. Thus, we obtain $A\in\mathfrak{u}(V,J_V,g_V)$.
    \end{itemize}

    Now let $X\in W$, and write \begin{align*}
        \d\omega(X,e_{2n-1},e_{2n})&=g([e_{2n},X],Je_{2n-1})+g([e_{2n-1},e_{2n}],JX)=g(DX,e_{2n})-g(De_{2n-1},JX)\\
        &=g(DX,e_{2n})-g(v_3,JX)-g(ae_{2n-1},JX).
    \end{align*}

    Again, there are several cases: \begin{itemize}
        \item If $X=e_1\in\frh^\perp$, then $\d\omega(e_1,e_{2n-1},e_{2n})=h_1+a$ is zero if and only if $h_1=-a$.
        \item If $X=e_2\in J\frh^\perp$, then $\d\omega(e_2,e_{2n-1},e_{2n})=h_2$ is zero if and only if $h_2=0$.
        \item If $X\in V$, then $\d\omega(X,e_{2n-1},e_{2n})=w_1^tX-g(v_3,JX)=w_1^tX+g(Jv_3,X)$ is zero for all $X$ if and only if $w_1^t=(-J_Vv_3)^{\flat_V}$. Notice that $(-J_Vv_3)^{\flat_V}J_v=-v_3^{\flat_V}$.
    \end{itemize}

    Combining all these conditions and setting $v:=v_3$ we obtain $D$ as in the statement.
\end{proof}

\begin{remark}\label{remark:freibertisotropic}
    Proposition~\ref{prop:characterization_isotropic_case} is consistent with \cite[Theorem 2.43]{Freibert24_TFHS}; the fact that the analogous of the component $v$ is missing in \cite[Example 2.45]{Freibert24_TFHS} appears to be a mistake. Indeed, in the notation of \cite{Freibert24_TFHS}, the Lie algebra $\lie k_{\lie u(p,q)}$ is the Lie algebra of the Lie group whose generic element is given in \eqref{eqn:stabh}.
\end{remark}

\begin{example}\label{example:4d}
    Four-dimensional pseudo-Kähler Lie algebras have been classified by Ovando in \cite{Ova06}. In her notation, the almost abelian ones correspond to $\mathfrak{rh}_3$, $\mathfrak{rr}_{3,0}$, $\mathfrak{rr}'_{3,0}$, $\mathfrak{r}_{4,-1,-1}$ and $\mathfrak{r}'_{4,0,\delta}$ for $\delta>0$. We recover them as follows. Let us first apply the ``non-isotropic'' construction of Proposition~\ref{prop:characterization_non-isotropic_case}. Assuming $\frg$ to be four-dimensional, the subspace $\frh_1$ is two-dimensional; we equip it with a positive-definite Kähler structure $(J_1,g_1)$. Then $A\in\mathfrak{u}(\frh_1,J_1,g_1)\cong\mathfrak{u}(1)$ and the derivation \eqref{eq:D_positive_g0} takes the form $$D=\begin{pmatrix}
        0&\lambda&0\\
        -\lambda&0&0\\
        0&0&a
    \end{pmatrix},$$ where $\lambda,a\in\R$. We get the following: \begin{itemize}
        \item If $\lambda=0$ and $a\neq0$, then $\frg\cong\mathfrak{rr}_{3,0}$.
        \item If $\lambda\neq0$ and $a=0$, then $\frg\cong\mathfrak{rr}'_{3,0}$.
        \item If $\lambda,a\neq0$, then $\frg\cong\mathfrak{r}'_{4,0,\delta}$.
    \end{itemize}

    We now apply the ``isotropic'' construction of Proposition~\ref{prop:characterization_isotropic_case}. To obtain $\frg$ of dimension four, the subspace $V$ is zero-dimensional, hence the derivation \eqref{eq:D_zero_g0} takes the form $$D=\begin{pmatrix}
        -a&0&c_1\\
        0&-a&c_2\\
        0&0&a
    \end{pmatrix},$$ where $c_1,c_2,a\in\R$. We get the following: \begin{itemize}
        \item If $a\neq0$, then $\frg\cong\mathfrak{r}_{4,-1,-1}$ for every value of $c_1,c_2$.
        \item If $a=0$, then $\frg\cong\mathfrak{rh}_3$ for every value of $c_1,c_2$ satisfying $c_1^2+c_2^2\neq0$.
    \end{itemize}
\end{example}

\begin{example}\label{example:6d_nilpotent}
    Six-dimensional nilpotent Lie algebras admitting a pseudo-Kähler structure have been classified in \cite{CFU04}. Only three out of the thirteen are almost abelian, and they appear in the classification as $\frh_6$, $\frh_8$ and $\frh_{10}$. We recover them as follows. Let us first consider the construction with $\frh_0$ non-isotropic. In order to get a nilpotent Lie algebra, we need the derivation $D$ as in \eqref{eq:D_positive_g0} to be nilpotent, which is equivalent to $A\in\mathfrak{u}(\frh_1,J_1,g_1)$ being nilpotent and $a=0$. If $g_1$ is a definite metric, then $A=0$, thus we need to take $g_1$ of neutral signature, say $g_1=e^1\otimes e^1+e^2\otimes e^2-e^3\otimes e^3-e^4\otimes e^4$. Take $A\in\mathfrak{u}(\frh_1,J_1,g_1)\cong\mathfrak{u}(1,1)$ as $$A=\begin{smallpmatrix}
        0 & 1 & 1 & 0 \\
        -1 & 0 & 0 & 1 \\
        1 & 0 & 0 & -1 \\
        0 & 1 & 1 & 0
    \end{smallpmatrix}.$$

    Then $\frg=\R^5\rtimes_D\R$ is a six-dimensional two-step nilpotent pseudo-Kähler Lie algebra isomorphic to $\frh_6$.\medskip
    
    Consider now the construction with $\frh_0$ isotropic. Since $\frg$ is six-dimensional, the subspace $V$ is two-dimensional and we equip it with a positive-definite Kähler structure $(J_V,g_V)$. In order to get a nilpotent Lie algebra, we need the derivation $D$ as in \eqref{eq:D_zero_g0} to be nilpotent, which is equivalent to $A\in\mathfrak{u}(V,J_V,g_V)$ being nilpotent and $a=0$. Since $g_V$ is positive-definite, we have $A=0$. We get the following: \begin{itemize}
        \item If $v\neq0$, then $\frg=\R^5\rtimes_D\R\cong\frh_{10}$ for every value of $c_1,c_2$.
        \item If $v=0$, then $\frg=\R^5\rtimes_D\R\cong\frh_8$ for every value of $c_1,c_2$ satisfying $c_1^2+c_2^2\neq0$.
    \end{itemize}
\end{example}

We conclude this section with a result about the cohomology of nilpotent almost abelian pseudo-Kähler Lie algebras.\medskip

For any Lie algebra $\frg$, let $H^{\bullet}(\g)$ denote the Chevalley-Eilenberg cohomology, which can be identified with the cohomology of the complex of left-invariant forms on the associated Lie group. Writing $b_1(\frg):=\dim H^1(\frg)$, the relation $$b_1(\frg)=\dim\frg-\dim[\frg,\frg]$$ holds by construction. Given an almost abelian Lie algebra $\frg$, one has $b_1(\frg)\geq1$, whilst for a nilpotent Lie algebra $\frg$ one has $b_1(\frg)\geq2$. It is conceivable that the existence of an invariant complex structure may increase this lower bound. Indeed, in \cite[Theorem 5.8]{LU25} the authors show that $b_1(M)\geq3$ for every pseudo-Kähler nilmanifold $M$ with an invariant complex structure, up to complex dimension four. We show that the same restriction holds, in any dimension, for almost abelian nilmanifolds equipped with an invariant pseudo-Kähler structure.

\begin{proposition}\label{prop:b1>2_in_PK}
    Let $(\frg,J,g)$ be a nilpotent almost abelian pseudo-Kähler Lie algebra. Then $b_1(\frg)\geq3$.
\end{proposition}

\begin{proof}
    If $\frh_0$ is non-isotropic, we can write $\frg=\R^{2n-1}\rtimes_D\R$ and the derivation $D$ takes the form \eqref{eq:D_positive_g0} with $A\in\mathfrak{u}(\frh_1,J_1,g_1)$ and $a\in\R$. The Lie algebra $\frg$ is nilpotent if and only if $D$ is nilpotent, which is equivalent to $A$ being nilpotent and $a=0$. In particular, this implies that $A$ has non-maximal rank, and $\rank D<2n-2$. Thus \begin{equation}\label{eqn:bettiinequality}
        b_1(\frg)=2n-\rank D >2.
    \end{equation}

    If $\frh_0$ is isotropic, we can also write $\frg=\R^{2n-1}\rtimes_D\R$, where the derivation $D$ now takes the form \eqref{eq:D_zero_g0}. As before, the Lie algebra $\frg$ is nilpotent if and only if $D$ is nilpotent, which is equivalent to $A$ being nilpotent and $a=0$. This implies that $De_1=De_2=0$, thus $\dim\ker D\geq2$ and again we obtain \eqref{eqn:bettiinequality} by the rank-nullity theorem.
\end{proof}


\section{Curvature of almost abelian pseudo-Kähler Lie algebras}\label{section:curvature}

In this section we study the curvature properties of almost abelian pseudo-Kähler Lie algebras. We use the characterization of these algebras given by Proposition~\ref{prop:characterization_non-isotropic_case} in the non-isotropic case, and by Proposition~\ref{prop:characterization_isotropic_case} in the isotropic case.\medskip

Throughout the section, $\nabla$ will denote the Levi-Civita connection on a pseudo-Kähler Lie algebra $(\frg,J,g)$. For $X\in\g$, the covariant derivative gives a linear operator $\nabla_X\colon\g\to\g$; it will be convenient to use these operators to characterize $\nabla$.\medskip

Let us start with the non-isotropic case.

\begin{proposition}\label{prop:curvature_definite_g0}
    Let $(\frg,J,g)$ be an almost abelian pseudo-Kähler Lie algebra with $\frh_0$ non-isotropic, so that $\g=\h\rtimes_D\R$ with $D$ as in \eqref{eq:D_positive_g0} and $(J,g)$ takes the form \eqref{eq:non-isotropic_pseudo-Hermitian}. Then the Levi-Civita connection is given by \begin{enumerate}
        \item $\nabla_X=0$ for $X\in\frh_1$.
        \item $\nabla_{e_{2n-1}}=a(e^{2n-1}\otimes e_{2n}-e^{2n}\otimes e_{2n-1})$.
        \item $\nabla_{e_{2n}}X=AX$ for $X\in\frh_1$ and $\nabla_{e_{2n}}e_{2n-1}=0=\nabla_{e_{2n}}e_{2n}$.
    \end{enumerate}

    Moreover, the only nonzero curvature term is given by $$R(e_{2n-1},e_{2n})=a^2(e^{2n-1}\otimes e_{2n}-e^{2n}\otimes e_{2n-1})=a\nabla_{e_{2n-1}}$$ and the Ricci curvature is given by $$\Ric=-a^2(e^{2n-1}\otimes e^{2n-1}+e^{2n}\otimes e^{2n}).$$
\end{proposition}

\begin{proof}
    The Levi-Civita connection $\nabla$ of the metric $g$ on $\frg$ is given by the Koszul formula: \begin{equation}\label{eq:Koszul}
        2g(\nabla_XY,Z)=g([X,Y],Z)-g([Y,Z],X)+g([Z,X],Y)
    \end{equation} for $X,Y,Z\in\frg$. If $X,Y,Z\in\frh$ then $g(\nabla_XY,Z)=0$. Now we compute $$2g(\nabla_XY,e_{2n})=-g([Y,e_{2n}],X)+g([e_{2n},X],Y)=g(DY,X)+g(DX,Y).$$

    Recall that $\frh=\frh_1\oplus J\frh^\perp$ and $J\frh^\perp=\Span{e_{2n-1}}$. If $X,Y\in\frh_1$, then $$2g(\nabla_XY,e_{2n})=g(DY,X)+g(DX,Y)=g(AY,X)+g(AX,Y)=g((A+A^*)X,Y)=0.$$

    Hence $\nabla_XY=0$ for all $X,Y\in\frh_1$. If $X\in\frh_1$, then $g(\nabla_Xe_{2n-1},e_{2n})=0$ and $g(\nabla_Xe_{2n},e_{2n})=0$. Therefore $\nabla_X=0$ for all $X\in\frh_1$.\medskip

    It remains to compute the terms $\nabla_{e_{2n-1}}$ and $\nabla_{e_{2n}}$. First note that $\nabla_{e_{2n-1}}X=\nabla_Xe_{2n-1}=0$ for all $X\in\frh_1$. Next, using that $De_{2n-1}=ae_{2n-1}$, we obtain $$g(\nabla_{e_{2n-1}}e_{2n-1},e_{2n})=ag(e_{2n-1},e_{2n-1})=a\varepsilon_{2n-1}=a\varepsilon_{2n},$$ thus $\nabla_{e_{2n-1}}e_{2n-1}=ae_{2n}$. Similarly, one can show that $\nabla_{e_{2n-1}}e_{2n}=-ae_{2n-1}$. Therefore, we can conclude that $\nabla_{e_{2n-1}}=ae^{2n-1}\otimes e_{2n}-ae^{2n}\otimes e_{2n-1}$. By similar computations, one obtains $g(\nabla_{e_{2n}}X,Y)=g(AX,Y)$ for $X,Y\in\frh_1$ and zero otherwise.\medskip

    For the curvature, it is clear that $R(X,Y)=R(X,e_{2n-1})=0$ and $$R(X,e_{2n})=-\nabla_{[X,e_{2n}]}=\nabla_{DX}=\nabla_{AX}=0$$ for $X,Y\in\frh_1$. Finally, $R(e_{2n-1},e_{2n})=[\nabla_{e_{2n-1}},\nabla_{e_{2n}}]-\nabla_{[e_{2n-1},e_{2n}]}=\nabla_{ae_{2n-1}}$.\medskip

    For the Ricci curvature, consider the orthonormal basis $\{e_1,\ldots,e_{2n}\}$. For $v,w\in\frg$, the Ricci curvature is given by \begin{align*}
        \Ric(v,w)&=\sum_{i=1}^{2n}\varepsilon_ig(R(e_i,v)w,e_i)\\
        &=\sum_{i=1}^{2n-2}\varepsilon_ig(R(e_i,v)w,e_i)+\varepsilon_{2n-1}g(R(e_{2n-1},v)w,e_{2n-1})+\varepsilon_{2n}g(R(e_{2n},v)w,e_{2n})\\
        &=\varepsilon_{2n-1}g(R(e_{2n-1},v)w,e_{2n-1})+\varepsilon_{2n}g(R(e_{2n},v)w,e_{2n}),
    \end{align*} since $R(e_i,v)=0$ for $i=1,\ldots,2n-2$. Using the above expression for the Ricci curvature we obtain that $\Ric(e_{2n-1},e_{2n-1})=-a^2=\Ric(e_{2n},e_{2n})$ and all the other terms are zero.
\end{proof}

Recall that a pseudo-Riemannian Lie algebra $(\frg,g)$ is an \emph{algebraic Ricci soliton} if there exists some $\lambda\in\R$ and a derivation $\delta\in\Der(\frg)$ such that $\ric=\lambda\Id+\delta$, where $\ric$ is the Ricci endomorphism defined by $\Ric=g(\ric\cdot,\cdot)$.\medskip

Let $(\frg,J,g)$ be an almost abelian pseudo-Kähler Lie algebra with $\frh_0$ non-isotropic. It is not difficult to check that the endomorphism $$\delta=\operatorname{diag}(a^2,\ldots,a^2,0,0)\in\End(\frg),$$ defined with respect to the decomposition $\frg=\frh_1\oplus J\frh^\perp\oplus\frh^\perp$, is a derivation of $\frg$. Using Proposition \ref{prop:curvature_definite_g0} we see that $\ric=-\varepsilon_{2n}a^2\Id+\varepsilon_{2n}\delta$ holds. Thus we conclude the following.

\begin{corollary}\label{cor:Ricci_soliton}
    Let $(\frg,J,g)$ be an almost abelian pseudo-Kähler Lie algebra with $\frh_0$ non-isotropic. Then $(\frg,J,g)$ is an algebraic Ricci soliton.
\end{corollary}

We consider now the isotropic case.

\begin{proposition}\label{prop:curvature_isotropic_h0}
    Let $(\frg,J,g)$ be an almost abelian pseudo-Kähler Lie algebra with $\frh_0$ isotropic, so that $\g=\h\rtimes_D\R$ with $D$ as in \eqref{eq:D_zero_g0} and $(J,g)$ takes the form \eqref{eqn:almosthermitianstructurezero}. Then the Levi-Civita connection is given by \begin{enumerate}
        \item $\nabla_{e_1}=\nabla_{e_2}=0$.
        \item $\nabla_X=0$ for $X\in V$.
        \item $\nabla_{e_{2n-1}}=-c_2(e^{2n-1}\otimes e_1+e^{2n}\otimes e_2)$.
        \item $\nabla_{e_{2n}}X=g(v,JX)e_1+g(v,X)e_2+AX$ for $X\in V$ and $$\begin{aligned}
            \nabla_{e_{2n}}e_1&=-ae_1,\\
            \nabla_{e_{2n}}e_2&=-ae_2,
        \end{aligned}\qquad\begin{aligned}
            \nabla_{e_{2n}}e_{2n-1}&=c_1e_1+v+ae_{2n-1},\\
            \nabla_{e_{2n}}e_{2n}&=c_1e_2+Jv+ae_{2n}.
        \end{aligned}$$
    \end{enumerate}

    Moreover, the only nonzero curvature term is given by $$R(e_{2n-1},e_{2n})=-3ac_2(e^{2n-1}\otimes e_1+e^{2n}\otimes e_2)=3a\nabla_{e_{2n-1}}$$ and the metric $g$ is Ricci-flat.
\end{proposition}

\begin{proof}
    The proof is analogous to that of Proposition \ref{prop:curvature_definite_g0}. We use the Koszul formula \eqref{eq:Koszul} to compute the covariant derivative operators in each of the four cases listed in the statement. We first note that $g(\nabla_XY,Z)=0$ for all $X,Y,Z\in\frh$.\medskip

    (1) We have $g(\nabla_{e_1}X,Y)=0$ for all $X,Y\in\frh$ and $g(\nabla_{e_1}e_{2n},X)=-g(e_{2n},\nabla_{e_1}X)$ for all $X\in\frg$. In particular $g(\nabla_{e_1}e_{2n},e_{2n})=0$. Therefore, we only need to compute the following terms: \begin{itemize}
        \item $2g(\nabla_{e_1}e_1,e_{2n})=-g([e_1,e_{2n}],e_1)+g([e_{2n},e_1],e_1)=2g(De_1,e_1)=-2ag(e_1,e_1)=0$ since $e_1$ is null.
        \item $2g(\nabla_{e_1}e_2,e_{2n})=-g([e_2,e_{2n}],e_1)+g([e_{2n},e_1],e_2)=g(De_2,e_1)+g(De_1,e_2)=-2ag(e_1,e_2)=0$ since $e_1$ and $e_2$ are orthogonal.
        \item Let $X\in V$, then $2g(\nabla_{e_1}X,e_{2n})=-g([X,e_{2n}],e_1)+g([e_{2n},e_1],X)=g(DX,e_1)+g(De_1,X)=0$ since $DX\in\frh$ and $\frh$ is orthogonal to $e_1$.
        \item $2g(\nabla_{e_1}e_{2n-1},e_{2n})=-g([e_{2n-1},e_{2n}],e_1)+g([e_{2n},e_1],e_{2n-1})=g(De_{2n-1},e_1)+g(De_1,e_{2n-1})=0$ since $De_{2n-1}\in\frh$ and $\frh$ is orthogonal to $e_1$.
    \end{itemize}

    We thus conclude that $\nabla_{e_1}=0$. To compute $\nabla_{e_2}$ we proceed in a similar way. The only difference is that $g(e_2,e_{2n-2})=-1$. Nevertheless, $$2g(\nabla_{e_2}e_{2n-1},e_{2n})=g(De_{2n-1},e_2)+g(De_2,e_{2n-1})=ag(e_{2n-1},e_2)+g(-ae_2,e_{2n-1})=0.$$

    Then $\nabla_{e_2}=0$ as well.\medskip

    (2) Let $X\in V$. Then for all $Y\in V$ we have $$2g(\nabla_XY,e_{2n})=g(DY,X)+g(DX,Y)=g(AY,X)+g(AX,Y)=g((A+A^*)X,Y)=0,$$ where we have used once again that $e_1$ and $e_2$ are orthogonal to $V$. We also have \begin{align*}
        2g(\nabla_Xe_{2n-1},e_{2n})&=g(De_{2n-1},X)+g(DX,e_{2n-1})=g(v,X)+g(v^tg_VXe_2,e_{2n-1})\\
        &=g(v,X)+g(v,X)g(e_2,e_{2n-1})=0,
    \end{align*} since $g(e_2,e_{2n-1})=-1$. By similar arguments to the previous case we conclude that $\nabla_X=0$ for $X\in V$.\medskip

    (3) First note that $\nabla_{e_{2n-1}}X=\nabla_Xe_{2n-1}=0$ for all $X\in\frh^\perp\oplus J\frh^\perp\oplus V$. Now we compute $$g(\nabla_{e_{2n-1}}e_{2n-1},e_{2n})=g(De_{2n-1},e_{2n-1})=g(c_2e_2,e_{2n-1})=-c_2$$ and $g(\nabla_{e_{2n-1}}e_{2n},e_{2n-1})=-g(e_{2n},\nabla_{e_{2n-1}}e_{2n-1})=c_2$. All the other terms are zero, thus we conclude that $\nabla_{e_{2n-1}}=-c_2e^{2n-1}\otimes e_1-c_2e^{2n}\otimes e_2$.\medskip

    (4) Here we make use of the fact that $\nabla$ is torsion-free: \begin{itemize}
        \item $\nabla_{e_{2n}}e_1=\nabla_{e_1}e_{2n}+[e_{2n},e_1]=De_1=-ae_1$.
        \item $\nabla_{e_{2n}}e_2=\nabla_{e_2}e_{2n}+[e_{2n},e_2]=De_2=-ae_2$.
        \item $\nabla_{e_{2n}}X=\nabla_Xe_{2n}+[e_{2n},X]=DX=v^tg_VJ_VXe_1+v^tg_VXe_2+AX=g(v,JX)e_1+g(v,X)e_2+AX$ for all $X\in V$.
        \item $\nabla_{e_{2n}}e_{2n-1}=\nabla_{e_{2n-1}}e_{2n}+[e_{2n},e_{2n-1}]=-c_2e_2+De_{2n-1}=c_1e_1+v+ae_{2n-1}$.
    \end{itemize}

    Finally we compute $\nabla_{e_{2n}}e_{2n}$: \begin{itemize}
        \item $g(\nabla_{e_{2n}}e_{2n},e_1)=-g(e_{2n},\nabla_{e_{2n}}e_1)=-g(e_{2n},-ae_1)=a$.
        \item $g(\nabla_{e_{2n}}e_{2n},e_2)=-g(e_{2n},\nabla_{e_{2n}}e_2)=-g(e_{2n},-ae_2)=0$.
        \item $g(\nabla_{e_{2n}}e_{2n},X)=-g(e_{2n},\nabla_{e_{2n}}X)=-g(e_{2n},g(v,JX)e_1)=-g(v,JX)=g(Jv,X)$.
        \item $g(\nabla_{e_{2n}}e_{2n},e_{2n-1})=-g(e_{2n},\nabla_{e_{2n}}e_{2n-1})=-g(e_{2n},c_1e_1)=-c_1$.
    \end{itemize}

    Therefore we get $\nabla_{e_{2n}}e_{2n}=c_1e_2+Jv+ae_{2n}$.\medskip
    
    Let us compute now the Riemann curvature tensor of $g$. We immediately see that $R(X,Y)=0$ for all $X,Y\in\frh$. Now we have $$R(e_1,e_{2n})=-\nabla_{[e_1,e_{2n}]}=\nabla_{De_1}=-a\nabla_{e_1}=0,$$ and similarly $R(e_2,e_{2n})=R(X,e_{2n})=0$ for all $X\in V$. Using that $\nabla_{De_{2n-1}}=a\nabla_{e_{2n-1}}$ we get $$R(e_{2n-1},e_{2n})=\nabla_{e_{2n-1}}\nabla_{e_{2n}}-\nabla_{e_{2n}}\nabla_{e_{2n-1}}+a\nabla_{e_{2n-1}}.$$

    The above expression is not identically zero only when evaluated in $e_{2n-1}$ and $e_{2n}$. Indeed: \begin{align*}
        R(e_{2n-1},e_{2n})e_{2n-1}&=\nabla_{e_{2n-1}}\nabla_{e_{2n}}e_{2n-1}-\nabla_{e_{2n}}\nabla_{e_{2n-1}}e_{2n-1}+a\nabla_{e_{2n-1}}e_{2n-1}\\
        &=\nabla_{e_{2n-1}}(c_1e_1+v+ae_{2n-1})-\nabla_{e_{2n}}(-c_2e_1)+a(-c_2e_1)\\
        &=a\nabla_{e_{2n-1}}e_{2n-1}+c_2\nabla_{e_{2n}}e_1-ac_2e_1\\
        &=-3ac_2e_1,\\
        R(e_{2n-1},e_{2n})e_{2n}&=R(e_{2n-1},e_{2n})Je_{2n-1}=JR(e_{2n-1},e_{2n})e_{2n-1}\\
        &=-3ac_2e_2.
    \end{align*}
    Therefore we conclude that $R(e_{2n-1},e_{2n})=-3ac_2(e^{2n-1}\otimes e_1+e^{2n}\otimes e_2)=3a\nabla_{e_{2n-1}}$.\medskip

    Let $E_1,\ldots,E_{2n}$ be an orthonormal basis of $\frg$, where $E_i=e_i$ for $i=3,\ldots,2n-2$ and $$E_1=\tfrac{1}{\sqrt{2}}(e_1+e_{2n}),\quad E_2=\tfrac{1}{\sqrt{2}}(e_2+e_{2n-1}),\quad E_{2n-1}=\tfrac{1}{\sqrt{2}}(e_2-e_{2n-1}),\quad E_{2n}=\tfrac{1}{\sqrt{2}}(e_1-e_{2n}).$$

    Since $R(X,w)=0$ for all $X\in\frh^\perp\oplus J\frh^\perp\oplus V$, $\Ric(X,w)=\sum_{i=1}^{2n}\varepsilon_ig(R(E_i,X)w,E_i)=0$. Moreover, since pseudo-K\"ahler metrics satisfy $\Ric(X,JX)=0$, we also have that $\Ric(e_{2n-1},e_{2n})=\Ric(e_{2n-1},Je_{2n-1})=0$. Now we compute \begin{align*}
        \Ric(e_{2n-1},e_{2n-1})&=\sum_{i=1}^{2n}\varepsilon_ig(R(E_i,e_{2n-1})e_{2n-1},E_i)\\
        &=\varepsilon_1g(R(E_1,e_{2n-1})e_{2n-1},E_1)+\varepsilon_{2n}g(R(E_{2n},e_{2n-1})e_{2n-1},E_{2n}),
    \end{align*} where $\varepsilon_1=g(E_1,E_1)=1$ and $\varepsilon_{2n}=g(E_{2n},E_{2n})=-1$. The first term gives $\frac{3}{2}ac_2$ and the second gives $-\frac{3}{2}ac_2$. Therefore $\Ric(e_{2n-1},e_{2n-1})=0$. Finally,
    \[\Ric(e_{2n},e_{2n})=\Ric(Je_{2n-1},Je_{2n-1})=\Ric(e_{2n-1},e_{2n-1})=0.\] This concludes the proof that the metric $g$ is Ricci-flat.
\end{proof}

It is known that every nilpotent pseudo-Kähler Lie algebra is Ricci-flat \cite[Lemma 6.4]{FinoPartonSalamon} (see \cite[Lemma 1.7]{Conti_Rossi_SegnanDalmasso_2023} for a more direct proof) and that so is every unimodular pseudo-Kähler Lie algebra with abelian complex structure \cite[Corollary 4.7]{BS12}. More results about Ricci-flat pseudo-Kähler Lie algebras can be found in \cite{Yamada12}.

\begin{remark}\label{remark:freibertflat}
    In the isotropic case, the condition $c_2=0$ characterizes the case when $D$ is contained in the subspace $\tilde{\lie k}_{\lie u(p,q)}$ of \cite{Freibert24_TFHS}. By the results therein, this is a sufficient condition for the metric to be flat; it can be seen from Proposition~\ref{prop:curvature_isotropic_h0} that this condition is not necessary.
\end{remark}

Recall that a Lie algebra $\frg$ is \emph{unimodular} if $\Tr(\ad(X))=0$ for all $X\in\frg$.

\begin{corollary}\label{cor:flat_almost_ab_pK}
    Let $(\frg,J,g)$ be a unimodular almost abelian pseudo-Kähler Lie algebra. Then $g$ is flat.
\end{corollary}

\begin{proof}
    Suppose that $\frh_0$ is non-isotropic. Then $\Tr(D)=\Tr(A)+a=a$ since $\Tr(A)=0$ for any $A\in\mathfrak{u}(\frh_1,J_1,g_1)$. If $\frg$ is unimodular, then $a=0$ and the metric $g$ is flat by Proposition \ref{prop:curvature_definite_g0}. In fact in this case the converse is also true, i.e.\ if the metric $g$ is flat, then $\frg$ is unimodular. Now suppose that $\frh_0$ is isotropic. Then $\Tr(D)=-a$, thus if $\frg$ is unimodular, then $g$ is flat by Proposition \ref{prop:curvature_isotropic_h0}.
\end{proof}

\begin{remark}\label{remark:hypersymplectic}
    Recall that a hypersymplectic structure on a Lie algebra is a triple $(J,E,g)$ such that $(J,g)$ defines a pseudo-K\"ahler structure, $(E,g)$ defines a para-K\"ahler structure, and $JE=-EJ$. It was conjectured in \cite[Conjecture 4.24]{CG25} that every two-step nilpotent hypersymplectic Lie algebra is flat. Since every hypersymplectic Lie algebra is in particular pseudo-Kähler, Corollary \ref{cor:flat_almost_ab_pK} gives a positive answer to this conjecture in the nilpotent almost abelian case.
\end{remark}

If we focus on the nilpotent case, we have the following.

\begin{corollary}\label{cor:completeness}
    Let $(\frg,J,g)$ be a nilpotent almost abelian pseudo-Kähler Lie algebra. Then $g$ is complete.
\end{corollary}

\begin{proof}
    Since $\frg$ is nilpotent, in particular unimodular, the metric $g$ is flat by Corollary \ref{cor:flat_almost_ab_pK}. Therefore, the flat torsion-free Levi-Civita connection $\nabla$ of $g$ defines a \emph{left-symmetric structure} on $\frg$. By the results of \cite{Segal92}, the completeness of $\nabla$ is equivalent to $\lambda\colon\frg\to\mathfrak{gl}(\frg)$, defined by $\lambda(X)=\nabla_X$, being a nilpotent endomorphism for every $X\in\frg$. By looking at Propositions \ref{prop:curvature_definite_g0} and \ref{prop:curvature_isotropic_h0} we see that this is the case.
\end{proof}

\begin{remark}\label{remark:nonuniqueness}
    It was shown in \cite{AABRW25} that complex structures on a nilpotent almost abelian Lie algebra are unique up to isomorphism. However, pseudo-Kähler structures on a nilpotent almost abelian Lie algebra are not unique up to isometric isomorphism. Indeed, consider the pseudo-K\"ahler Lie algebras obtained by setting $a$ and $A$ to zero in \eqref{eq:D_zero_g0} and choosing $v\neq0$. As observed in Example~\ref{example:6d_nilpotent}, the resulting Lie algebras are isomorphic to $\frh_{10}$. However, Proposition~\ref{prop:curvature_isotropic_h0} shows that if $c_2=0$, then $\nabla_{e_i}=0$ for $i=1,\ldots,5$; but if $c_2\neq0$, then $\nabla_{e_5}$ is linearly independent with $\nabla_{e_6}$. Hence, the rank of the linear map $\nabla\colon\frg\to\mathfrak{gl}(\frg)$ is different in each case, so the corresponding metrics are not related by an isometric isomorphism (see also \cite[Proposition 3.9]{CFU04}).
\end{remark}

\begin{remark}\label{remark:uniquenessuptoisometries}
    We have seen in Corollaries \ref{cor:flat_almost_ab_pK} and \ref{cor:completeness} that every pseudo-Kähler structure on a nilpotent almost abelian Lie algebra is flat and complete, so that the resulting pseudo-Riemannian manifolds are isometric. Therefore, the non-uniqueness observed in Remark~\ref{remark:nonuniqueness} depends on the fact that one only considers isometries that respect the Lie algebra structure.
\end{remark}


\section{Classification of almost abelian pseudo-Kähler Lie algebras}\label{section:classification}

In this section we classify almost abelian pseudo-Kähler Lie algebras in every dimension. For that, in Subsection \ref{section:orbits} we recall the classification of adjoint orbits of $\mathrm{U}(p,q)$ obtained in \cite{Burgoyne_Cushman_1977}. Then in Subsection \ref{section:classification_every_dim} we obtain the aforementioned classification, which we describe more explicitly in dimensions six and eight, respectively in Subsections \ref{section:6d_classification} and \ref{section:8d_classification}.

\subsection{Orbits in \texorpdfstring{$\mathrm{U}(p,q)$}{U(p,q)}}\label{section:orbits}

Here we recall the classification of adjoint orbits of $\mathrm{U}(p,q)$ from \cite{Burgoyne_Cushman_1977}, or equivalently, of skew-Hermitian endomorphisms up to unitary transformations. We will adapt the notation slightly; in particular, we will favour representations over equivalence classes.\medskip

We consider a complex vector space $V$ endowed with a Hermitian form $h$ and a skew-Hermitian endomorphism $A\colon V\to V$. We say that $(A,V,h)$ is equivalent to $(A',V',h')$ if there is an isomorphism $\phi\colon V\to V'$ such that $\phi\circ A=A'\circ\phi$ and $h'(\phi\cdot,\phi\cdot)=h$.\medskip

For $m\geq0$, consider the Hermitian form on $\C^{m+1}$ defined by the matrix $$h_m=\begin{pmatrix}
    &&&&1\\
    &&&-1&\\
    &&\iddots\\
    &-1\\
    1
\end{pmatrix}\,(m\text{ even}),\qquad h_m=\begin{pmatrix}
    &&&&i\\
    &&&-i&\\
    &&\iddots\\
    &i\\
    -i
\end{pmatrix}\,(m\text{ odd}).$$

For $\zeta\in\C$ and $m\geq0$, consider the Jordan block of size $m+1$ \begin{equation}\label{eqn:jordanblock}
    A_{\zeta,m}:=\begin{pmatrix}
        \zeta & 1 \\
        & \ddots & 1 \\
        & & \zeta
    \end{pmatrix}\in\End(\C^{m+1}).
\end{equation}

For $\zeta$ imaginary, this gives a skew-Hermitian endomorphism; set $$\Delta_m^\pm(\zeta):=( A_{\zeta,m},\C^{m+1},\pm h_m).$$

For $\zeta$ not imaginary, we can obtain a skew-Hermitian endomorphism by pairing two Jordan blocks; set $$\Delta_m(\zeta,-\overline\zeta):=\bigl( A_{\zeta,m}\oplus A_{-\overline \zeta,m},\C^{m+1}\oplus\C^{m+1},h_{2m+1}\bigr).$$

We will say that $(A,V,h)$ is a \emph{block} if it coincides with some $\Delta_m^\pm(\zeta)$ or $\Delta_m(\zeta,-\overline\zeta)$; by construction, a block $\Delta$ consists of a vector space, a Hermitian form and a skew-Hermitian endomorphism; as a matter of notation, we will denote the three objects as $V(\Delta)$, $h(\Delta)$ and $A(\Delta)$. Since $\Delta_m(\zeta,-\overline\zeta)$ and $\Delta_m(-\overline\zeta,\zeta)$ are clearly equivalent, we will assume without loss of generality that $\Re\zeta>0$ in blocks of type $\Delta_m(\zeta,-\overline\zeta)$. With this assumption, it is clear that the blocks $\Delta_m^\pm(\zeta)$ and $\Delta_m(\zeta,-\overline\zeta)$ are pairwise non-equivalent.

\begin{theorem}[\cite{Burgoyne_Cushman_1977}]\label{thm:orbits_U(p,q)}
    Let $V$ be a complex vector space endowed with a Hermitian form $h$ and let $A\colon V\to V$ be skew-Hermitian relative to $h$. Then $V=W_1\oplus\cdots\oplus W_k$, where the sum is orthogonal, each $W_i$ is $A$-invariant and $(A|_{W_i},W_i, h|_{W_i})$ is equivalent to some $\Delta_m^\pm(\zeta)$, $\zeta\in i\R$, or $\Delta_m(\zeta,-\overline\zeta)$, $\Re\zeta>0$. The decomposition is unique up to reordering.
\end{theorem}

This result gives a classification of skew-Hermitian endomorphisms of $\C^n$ up to unitary transformations, or equivalently of adjoint orbits in $\lie{u}(p,q)$. Rather than use this result directly, we will introduce a language that provides a method to construct an explicit list of skew-Hermitian endomorphism, one in each orbit.\medskip

Consider the abelian semigroup $\mathcal{U}$ freely generated by blocks. Every element $t\in\mathcal{U}$ determines a vector space $V(t)$, a Hermitian form $h(t)$ and a skew-Hermitian endomorphism $A(t)$, where $$V(a_1\Delta_1+\cdots+a_k\Delta_k):=\underbrace{V(\Delta_1)\oplus \dots \oplus V(\Delta_1)}_{a_1}\oplus \dots \oplus \underbrace{V(\Delta_k)\oplus \dots \oplus V(\Delta_k)}_{a_k},$$ and $h(a_1\Delta_1+\cdots+a_k\Delta_k)$, $A(a_1\Delta_1+\cdots+a_k\Delta_k)$ are defined in the obvious way. In the following sections, we will think of $V(t)$ as a real vector space endowed with an almost pseudo-Hermitian structure $(J(t),g(t))$, where $J(t)$ is given by multiplication by $i$ and $g(t)$ is the real part of $h(t)$.

\begin{lemma}\label{lemma:changesign}
    Let $r\colon \mathcal{U}\to\mathcal{U}$ be the additive map defined on generators by $$r(\Delta^{\varepsilon}_m(\zeta))=\Delta^{\varepsilon'}_m(-\zeta),\quad r(\Delta_m(\zeta,-\overline\zeta))=\Delta_m(\overline\zeta,-\zeta),$$ where $\varepsilon'=(-1)^m\varepsilon$. Then for any $t\in\mathcal{U}$, $(-A(t),V(t),h(t))$ is equivalent to $(A(r(t)),V(r(t)),h(r(t)))$.
\end{lemma}

\begin{proof}
    For a block of type $\Delta=\Delta^{\pm}_m(\zeta)$, we see that $$-A_{\zeta,m}=KA_{-\zeta,m}K^{-1},$$ where $K$ is the diagonal matrix with $(-1)^{i}$ at position $(i,i)$. Since $Kh_m\tran K=(-1)^{m}h_m$, we conclude that $$r(\Delta^{\pm}_m(\zeta))=\begin{cases}
        \Delta^{\pm}_m(-\zeta)&\text{if }m\text{ even},\\
        \Delta^{\mp}_m(-\zeta)&\text{if }m\text{ odd}.\\
    \end{cases}$$

    For a block of type $\Delta_m(\zeta,-\overline\zeta)$, we apply the same argument to each block, and obtain that changing the sign of $A$ amounts to replacing $\zeta$ with $-\zeta$; interchanging the two blocks in order to maintain the condition $\Re\zeta>0$, we obtain $r(\Delta_m(\zeta,-\overline\zeta))=\Delta_m(\overline\zeta,-\zeta)$.
\end{proof}

Now consider the additive map \begin{equation}\label{eq:signature_map}
    s\colon \mathcal U\to\N\oplus\N
\end{equation} associating to every block the signature of the Hermitian form; to be explicit: $$s(\Delta_{2k-1}^\varepsilon(\zeta))=(k,k),\quad s(\Delta_{2k}^\varepsilon(\zeta))=\begin{cases}
    (k+1,k)&\mbox{if }(-1)^k=+\varepsilon\\
    (k,k+1)&\mbox{if }(-1)^k=-\varepsilon\\
\end{cases},\quad s(\Delta_m(\zeta,-\overline\zeta))=(m+1,m+1),$$ where $\varepsilon\in\{-1,+1\}$. Then every $\mathrm{U}(p,q)$-orbit in $\lie u(p,q)$ can be identified with an element $t\in\mathcal{U}$ with $s(t)=(p,q)$. We will set $$\mathcal{U}_{p,q}:=s^{-1}(p,q).$$

In the classification of next section we will encounter a symmetry of the type $v\mapsto v+A(t)u$ (see Lemma \ref{lemma:orbitsarbitrarydimension}), where $t\in\mathcal{U}_{p,q}$, $u,v\in V(t)$; accounting for this symmetry, the classification will require fixing a complement $W$ of $\Im A(t)$ in $V(t)$ and an element in $W$ in each $\mathrm{U}(p,q)$-orbit. To that end, for each $\Delta^\pm_m(0)$, we declare $v(\Delta^\pm_m(0))$ to be the last element in the standard basis of $V(\Delta^\pm_m(0))=\C^{m+1}$, so that $$V(\Delta^\pm_m(0))=\SpanC{A^k(v(\Delta^\pm_m(0)))\mid k=0,\dotsc,m}_.$$ If $t=n\Delta$ and $\Delta$ is a nilpotent block, we view the vector $v(\Delta)$ as an element of $V(t)$ by $$v(\Delta)=\bigl(v(\Delta),0, \dotsc, 0\bigr)\in\underbrace{V(\Delta)\oplus \dots \oplus V(\Delta)}_{n}.$$

Fix a type $t$ in $\mathcal{U}$. Let $N_t$ be the set of nilpotent blocks appearing in $t$, i.e.\ $$N_t:=\{\Delta=\Delta^{\varepsilon}_m(0)\mid m\geq0, \varepsilon=\pm1, t-\Delta\in\mathcal{U}\}.$$

Let $X_t$ denote the set of maps from $N_t$ to nonnegative real numbers $\R_{\geq0}$ with the property that for all $m$ \begin{equation}\label{eq:condition_Xt}
    x(\Delta_m^+(0))x(\Delta_m^-(0))=0\quad\text{or}\quad x(\Delta_m^+(0))=1=x(\Delta_m^-(0)),
\end{equation} where $x(\Delta)$ is taken to be $0$ when $\Delta$ is not in $N_t$. Notice that $X_t$ does not depend on the multiplicities of the blocks. For instance, if $t=a\Delta_0^+(0)+b\Delta_0^-(0)$ with $a,b$ positive, then an element of $X_t$ can be identified with a pair $(x_0^+, x_0^-)\in(\R_{\geq0})^2$, irrespective of the values of $a,b$, and either $x_0^+$ is zero, $x_0^-$ is zero, or $x_0^+=1=x_0^-$. Every element of $V(t)$ is in the same $\mathrm{U}(a,b)$-orbit as an element $x_0^+v(\Delta_0^+(0))+x_0^-v(\Delta_0^-(0))$ with $(x_0^+,x_0^-)$ in $X_t$, motivating the restriction of \eqref{eq:condition_Xt}.

\begin{lemma}\label{lemma:parametrizecomplement}
    Suppose that $t\in\mathcal{U}_{p,q}$ contains a nilpotent block, and let $H$ be the stabilizer of $A(t)$ in $\mathrm{U}(p,q)$. Then there is a complement $W$ of $\Im A(t)$ in $V(t)$ such that every element in $W$ is in the same $H$-orbit as a uniquely determined vector \begin{equation}\label{eqn:vofx}
        v(x)=\sum_{\Delta\in N_t} x(\Delta)v(\Delta),
    \end{equation} where $x\in X_t$.
\end{lemma}

\begin{proof}
    Let $W$ be the complex subspace of $V(t)$ spanned by $v(\Delta)$, as $\Delta$ ranges among nilpotent blocks in $t$. Since $A_{\zeta,m}$ is invertible for $\zeta\neq0$, $W$ is a complement of $\Im A(t)$.\medskip

    If $t=n\Delta$, then $W(t)$ is spanned by $$(v(\Delta),0,\dotsc, 0),\dotsc, (0,\dotsc, 0,v(\Delta));$$ since  $\mathrm{U}(n)$ acts on $V(n\Delta)\cong\C^n\otimes_\C V(\Delta)$ by unitary transformations that preserve $A(t)$, every element of $W(t)$ is in the same $H$-orbit as some $xv(n\Delta)$, with $x\in\C$. Moreover, one can multiply $v=v(\Delta^\pm_m(0))$ by an element of $\mathrm{U}(1)$ with a change of the unitary basis, though not rescale it, since $\abs{h(A^{m}(v),v)}=1$. Thus, we can assume that $x$ is a nonnegative real number.\medskip

    Similarly, if $t=\Delta^+_m(0)+\Delta^-_m(0)$, then $\mathrm{U}(1,1)$ acts on $V(\Delta^+_m(0))\oplus V(\Delta^-_m(0))$ by preserving $A(t)$. Since every element of $\C^2$ is in the same $\mathrm{U}(1,1)$-orbit as an element of the form $$\begin{pmatrix} x \\ 0 \end{pmatrix}, \begin{pmatrix} 0 \\ y \end{pmatrix}, \begin{pmatrix} 1\\ 1 \end{pmatrix}$$ with $x,y$ nonnegative real numbers, we see that every element of $W(t)$ is in the same $H$-orbit as $xv(\Delta^+_m(0))$, $yv(\Delta^-_m(0))$ or $v(\Delta^+_m(0))+v(\Delta^-_m(0))$.\medskip

    The general case is obtained by combining the two arguments. Uniqueness follows from the fact that $H$ preserves generalized eigenspaces of $A(t)$.
\end{proof}

\subsection{Classification in arbitrary dimensions}\label{section:classification_every_dim}

Here we use the language of blocks introduced in Subsection \ref{section:orbits} to classify the possible derivations $D$ in the construction of almost abelian pseudo-Kähler Lie algebras.\medskip

We begin by introducing seven families of Lie algebras $\g_0,\dotsc,\g_6$:

\begin{itemize}
    \item Given $t\in\mathcal{U}_{p,q}$, $a\in\R$, $\varepsilon=\pm1$, $n=p+q+1$, set $$D_0(t,a)=\begin{pmatrix}
        A(t)&0\\
        0&a
    \end{pmatrix}\in\End(\frh), \quad \frh=V(t)\oplus\Span{e_{2n-1}}\cong\R^{2n-1}.$$

    By Proposition~\ref{prop:characterization_non-isotropic_case}, the almost abelian Lie algebra $$\g_0(t,a,\varepsilon):=\frh\rtimes_{D_0(t,a)}\Span{e_{2n}}$$ has a pseudo-Kähler structure $$J=J(t)+e^{2n-1}\otimes e_{2n}-e^{2n}\otimes e_{2n-1}, \quad g=g(t)+\varepsilon (e^{2n-1}\otimes e^{2n-1}+e^{2n}\otimes e^{2n}).$$
    \item Given $t\in\mathcal{U}_{p,q}$, $n=p+q+2$, set $$\frh=\Span{e_1,e_2}\oplus V(t) \oplus \Span{e_{2n-1}}\cong\R^{2n-1}.$$

    Consider six families of endomorphisms of $\frh$,\begin{gather*}
    D_1(t,c_2)=\begin{smallpmatrix}
        -1 & 0 & 0 & 0\\
        0 & -1 & 0 & c_2\\
        0 & 0& A(t) & 0\\
        0 & 0 & 0 & 1
    \end{smallpmatrix},\quad D_2(t,x)=\begin{smallpmatrix}
        0 & 0 & -(J_Vv(x))^{\flat_V} & 0\\
        0 & 0 & v(x)^{\flat_V} & 0\\
        0 & 0& A(t) & v(x)\\
        0 & 0 & 0 & 0
    \end{smallpmatrix},\quad D_3(t,x)=\begin{smallpmatrix}
        0 & 0 & -(J_Vv(x))^{\flat_V} & 0\\
        0 & 0 & v(x)^{\flat_V} & 1\\
        0 & 0& A(t) & v(x)\\
        0 & 0 & 0 & 0
    \end{smallpmatrix},\\
    D_4(t,c_1)=\begin{smallpmatrix}
        0 & 0 & 0 & c_1\\
        0 & 0 & 0 & 1\\
        0 & 0& A(t) & 0\\
        0 & 0 & 0 & 0
    \end{smallpmatrix}, \quad D_5(t)=\begin{smallpmatrix}
        0 & 0 & 0 & 1\\
        0 & 0 & 0 & 0\\
        0 & 0& A(t) & 0\\
        0 & 0 & 0 & 0
    \end{smallpmatrix},\quad D_6(t)=\begin{smallpmatrix}
        0 & 0 & 0 & 0\\
        0 & 0 & 0 & 0\\
        0 & 0& A(t) & 0\\
        0 & 0 & 0 & 0
    \end{smallpmatrix},
    \end{gather*} where $c_1,c_2$ are real numbers and $x$ a nonzero element of $X_t$. By Proposition~\ref{prop:characterization_isotropic_case}, for any $D_i(t,\dots)$ in one of these families we have an almost abelian Lie algebra $$\g_i(t,\dots):=\frh\rtimes_{D_i(t,\dots)}\Span{e_{2n}}$$ with a pseudo-Kähler structure given by $$J=e^1\otimes e_2-e^2\otimes e_1+J(t)+e^{2n-1}\otimes e_{2n}-e^{2n}\otimes e_{2n-1}, \quad g=e^1\odot e^{2n}-e^2\odot e^{2n-1}+g(t).$$
\end{itemize}

By construction, each Lie algebra in the families $\g_0,\dotsc, \g_6$ is endowed with a fixed pseudo-Kähler structure.

\begin{theorem}\label{thm:classification}
    Up to unitary Lie algebra isomorphisms, every almost abelian pseudo-Kähler Lie algebra belongs to one of the families $\g_0,\dotsc, \g_6$. Given two distinct Lie algebras in any two families $\g_i$, $\g_j$, a unitary Lie algebra isomorphism between them exists exactly in the following cases: \begin{itemize}
        \item between the abelian Lie algebras in the families $\g_0$ and $\g_6$;
        \item between $\g_0(t,a,\varepsilon)$ and $\g_0(r(t),-a,\varepsilon)$;
        \item between $\g_2(t,x)$ and $\g_2(t',x')$ if $A(t)=kA(t')$ and $x=k^2x'$ for some $k\in\R\setminus\{0\}$;
        \item between $\g_6(t)$ and $\g_6(t')$ if $A(t)=kA(t')$ for some $k\in\R\setminus\{0\}$.
    \end{itemize}
\end{theorem}

The proof will require a few lemmas.\medskip

In order to identify pairs of isomorphic Lie algebras in the families $\g_i$, we will repeatedly use the fact that the codimension one abelian ideal in an almost abelian Lie algebra is generally unique, with the important exception of the three-dimensional Heisenberg Lie algebra $\mathfrak{heis}_3$, defined by the brackets $$[e_1,e_2]=-e_3.$$

More precisely, we have the following result due to Freibert.

\begin{lemma}[{\cite[Proposition 1]{Freibert12}}]\label{lemma:freibert}
    Let $\phi\colon\h\rtimes_D\R\to \h'\rtimes_{D'}\R$ be an isomorphism between two almost abelian Lie algebras. Then $\phi(\h)\subset\h'$, unless the Lie algebras are abelian or isomorphic to $\mathfrak{heis}_3\oplus\R^{k}$ for some $k$.
\end{lemma}

Applying this criterion, we immediately obtain a classification of the case  where $\frh_0$ is non-isotropic.

\begin{lemma}\label{lemma:classification_non-isotropic}
    Given an almost abelian pseudo-Kähler Lie algebra $(\g,J,g)$ with $\frh_0$ non-isotropic, there exist $t\in\mathcal{U}_{p,q}$, $a\in\R$, $\varepsilon=\pm1$, and a unitary Lie algebra isomorphism between $(\g,J,g)$ and $\g_0(t,a,\varepsilon)$. The parameters $(t,a)$ can be replaced with $(r(t),-a)$; outside of this ambiguity, $a$ and $t$ are determined uniquely.
\end{lemma}

\begin{proof}
    The fact that every almost abelian pseudo-Kähler Lie algebra with $\frh_0$ non-isotropic takes the form $\g_0(t,a,\varepsilon)$ follows from Proposition~\ref{prop:characterization_non-isotropic_case} and the classification of skew-Hermitian matrices.\medskip

    To prove uniqueness, suppose that there is an isometric isomorphism preserving the complex structure between the pseudo-Kähler Lie algebras $\h\rtimes_{D_0(t,a)}\Span{e_{2n}}$ and $\h'\rtimes_{D_0(t',a')}\Span{e_{2n}'}$, with obvious notation, both endowed with the pseudo-Kähler structure described in the statement. By Freibert's criterion (see Lemma \ref{lemma:freibert}), either $D_0(t,a)$ and $D_0(t',a')$ are nilpotent matrices of rank less than two, or the isomorphism maps $\h$ to $\h'$. In the first case, $D_0(t,a)$ and $D_0(t',a')$ are trivially zero. Otherwise, the isomorphism maps $e_{2n}$ to $\pm e_{2n}'$, and consequently $e_{2n-1}$ to $\pm e_{2n-1}$. Thus, $a'=\pm a$ and $A(t')$ is conjugated to $\pm A(t)$. By Lemma \ref{lemma:changesign}, this implies that either $t'=t$ and $a'=a$, or $t'=r(t)$ and $a'=-a$.
\end{proof}

In the case where $\frh_0$ is isotropic, we first need to determine the subgroup of unitary transformations that leaves invariant the codimension one abelian ideal $\frh$.

\begin{lemma}\label{lemma:orbitsarbitrarydimension}
    Let $\g=\frh\rtimes_D\R$ and let $(J,g)$ be the almost pseudo-Hermitian structure defined in \eqref{eqn:almosthermitianstructurezero}. The stabilizer of $\frh$ in $\mathrm{U}(\g,J,g)$ is \begin{equation}\label{eqn:stabh}
    \left\{\begin{pmatrix}
        x & 0 & -x(J_VC^{-1}u)^{\flat_V} & y & -\frac{1}{2}xg_V(u,u)\\
        0 &x & x(C^{-1}u)^{\flat_V} & \frac{1}{2}xg_V(u,u) & y\\
        0 & 0 & C & u & J_V u \\
        0 & 0 & 0 & \frac{1}{x} & 0 \\
        0 & 0 & 0 & 0 & \frac{1}{x}
        \end{pmatrix}\left|\ \ \begin{gathered}x\in\R\setminus\{0\}, y\in\R\\u\in V\\C\in\mathrm{U}(V,J_V,g_V)\end{gathered}\right.\right\}.
    \end{equation}
    
    If $D$ is as in \eqref{eq:D_zero_g0}, changing the frame by \eqref{eqn:stabh} modifies the components of $D$ by $$a\mapsto \frac{a}{x}, \quad c_1\mapsto \frac{c_1}{x^3} -\frac{2ay}{x^2} - \frac{2g(J_V v,u)}{x^2}+\frac{g(J_Vu, Au)}{x}, \quad c_2\mapsto \frac{c_2}{x^3},$$
    $$v\mapsto \frac1xC^{-1}Au+\frac1{x^2}C^{-1}v- \frac{a}{x} C^{-1}u , \quad A\mapsto \frac{1}{x}C^{-1}AC.$$
\end{lemma}

\begin{proof}
    Let $\Gamma$ be a matrix in the stabilizer. Since $\Gamma$ preserves the metric, it also preserves $$\frh^\perp=\Span{e_1,\dotsc, e_{2n-1}}^\perp=\Span{e_1}.$$

    Since $\Gamma$ commutes with $J$, $e_1$ and $e_2$ are eigenvectors with eigenvalue $x\neq0$. Therefore $$\frh_0^\perp=\Span{e_1,e_2}^\perp=\Span{e_1,\dotsc, e_{2n-2}}$$ is invariant under $\Gamma$, and then $$\Gamma=\begin{pmatrix}
        xI & A & B \\
        0 & C & D \\
        0 & 0 & E
    \end{pmatrix}\in\End(\frg),$$ where $I\colon \frh_0\to \frh_0$ is the identity map, $C$ and $E$ are endomorphisms of $V$ and $\Span{e_{2n-1},e_{2n}}$, respectively, and $$A\colon V\to\frh_0, \quad B\colon \Span{e_{2n-1},e_{2n}}\to \frh_0,  \quad D\colon \Span{e_{2n-1},e_{2n}}\to V$$ are linear maps. Imposing that $\Gamma$ commutes with $J$ implies that each of $A,B,C,D,E$ is $J$-linear in the appropriate sense. We can write the metric tensor as $$g=\begin{pmatrix}
        0 & 0 & -\omega \\
        0 & g_V & 0 \\
        \omega & 0 & 0
    \end{pmatrix},\quad\omega=\begin{pmatrix}
        0 & -1\\
        1 & 0
    \end{pmatrix}.$$

    Imposing that $\Gamma$ is orthogonal, i.e.\ $g=\tran \Gamma g\Gamma$, gives $$-\omega = -x\omega E, \quad \tran C g_V C=g_V, \quad -\tran A\omega E +\tran C g_V D=0, \quad -\tran B\omega E +\tran Dg_V D+\tran E\omega B=0.$$

    Thus $E=\frac1x I$, $C$ is $g_V$-orthogonal, and writing out $D=\begin{pmatrix} u & J_Vu\end{pmatrix}$, we have $$\tran A =-x\tran C g_V D\omega=-x\tran C g_V \begin{pmatrix}J_Vu & -u\end{pmatrix}=-x g_V C^{-1} \begin{pmatrix}J_Vu & -u\end{pmatrix}=x g_V \begin{pmatrix}-C^{-1}J_Vu & C^{-1}u\end{pmatrix},$$ thus $$A=\begin{pmatrix}
        -x (J_VC^{-1}u)^{\flat_V}\\
        x(C^{-1}u)^{\flat_V}
    \end{pmatrix}.$$

    Finally, since $B$ is $J$-linear, it is a linear combination of $I$ and $\omega$. The component along $\omega$ is determined by the equation $-\tran B\omega E +\tran Dg_V D+\tran E\omega B=0$, and the other component is free. Therefore, we get $B=\frac{1}{2}xg(u,u)\omega+yI$ for some $y\in\R$, and the matrix $\Gamma$ takes the form of \eqref{eqn:stabh}; moreover $$\Gamma^{-1}=\begin{pmatrix}
        \frac{1}{x} & 0 & (J_Vu)^{\flat_V} & -y & -\frac{1}{2}xg_V(u,u)\\
        0 & \frac{1}{x} & - u^{\flat_V} & \frac{1}{2}xg_V(u,u) & -y\\
        0 & 0 & C^{-1} & -xC^{-1}u & -xJ_VC^{-1}u\\
        0 & 0 & 0 & x & 0 \\
        0 & 0 & 0 & 0 & x
    \end{pmatrix}.$$

    If we set $E_i:=\Gamma e_i$, then $\lie h=\Span{E_1,\dotsc, E_{2n-1}}$ is abelian, and $$\ad(E_{2n})|_{\lie h}=\frac1x\ad(e_{2n})|_{\lie h},$$ which in the basis $\{E_1,\dotsc, E_{2n-1}\}$ is given by the matrix $\frac{1}{x}(\Gamma^{-1})|_{\frh}D\Gamma|_\frh$. In other words, we consider the (inverse) conjugate action of $\Gamma$ on $D$, which ignores the last row and column. After a straightforward computation we get $$(\Gamma^{-1})|_{\frh}D\Gamma|_\frh=\left(\begin{smallmatrix}
        -a & 0 & a(J_VC^{-1}u)^{\flat_V}-\frac1x(J_Vv)^{\flat_V} C+(J_Vu)^{\flat_V} AC & \frac{c_1}{x^2}-\frac{2ay}{x}-\frac{2g(J_V v,u)}{x}+g(J_Vu, Au)\\
        0 & -a & -a(C^{-1}u)^{\flat_V} +\frac{1}{x}v^{\flat_V} C-u^{\flat_V} AC&\frac{c_2}{x^2}-g(u,Au)\\
        0 & 0 & C^{-1}AC & C^{-1}Au+\frac{1}{x}C^{-1}v-aC^{-1}u\\
        0 & 0 & 0 & a
    \end{smallmatrix}\right).$$

    Using that $A$ is skew-Hermitian (giving $g(u,Au)=0$), extracting components and dividing by $x$ we obtain the statement.
\end{proof}

\begin{proof}[Proof of Theorem \ref{thm:classification}]
    Let $\g=\frh\rtimes_D\R$ be an almost abelian pseudo-Kähler Lie algebra with $\frh_0$ non-isotropic. Then Lemma \ref{lemma:classification_non-isotropic} shows that it belongs to the family $\frg_0$.\medskip

    Let $\g=\frh\rtimes_D\R$ be an almost abelian pseudo-Kähler Lie algebra with $\frh_0$ isotropic, and write $D$ as in Proposition~\ref{prop:characterization_isotropic_case}. We will show that up to the action of $\mathrm{U}(\g,J,g)$, the derivation $D$ is in one of the families $D_i$, $1\leq i\leq 6$. This amounts to acting on $D$ by the transformations of Lemma \ref{lemma:orbitsarbitrarydimension}. It is clear that $A$ can be taken up to conjugation in the Lie algebra of $\mathrm{U}(V,J_V,g_V)$. Thus, we may assume $A=A(t)$ and $V=V(t)$ for some $t\in\mathcal{U}_{p,q}$, with $p+q=n-2$.\medskip

    If $a\neq0$, then $A-a\Id$ is invertible since $A$ is skew-Hermitian; acting with an element of the stabilizer with $x=1$, $C=\Id$, $y=0$, we can map an arbitrary element to one with $v=0$. We can then act with an element with $y\neq0$ so that $c_1$ becomes zero, and rescale $a=1$ using an element with $x=a$. This makes $D$ of the form $D_1(t,c_2)$.\medskip

    If $a=0$, we can act with an element with nonzero $u$ to obtain $v$ in the complement of $\Im A$ introduced in Lemma \ref{lemma:parametrizecomplement}. If $v$ is zero, taking into account the simultaneous rescalings of $c_1,c_2$, we obtain the families $D_4(t,c_1)$, $D_5(t)$ and $D_6(t)$.\medskip

    If $v$ is nonzero, by Lemma \ref{lemma:parametrizecomplement} we can assume $v=v(x)$ for a nonzero element $x\in X_t$; notice that this case can only occur if $N_t$ is non-empty, i.e.\ $t$ contains a nilpotent block. Moreover, there is some $u\in\ker A$ such that $g(v,Ju)\neq0$, and this allows us to obtain $c_1=0$. Since $v$ and $c_2$ can be simultaneously rescaled, this gives the families $D_2(t,x)$ and $D_3(t,x)$.\medskip

    Now suppose that we have two distinct almost abelian Lie algebras in two families $\g_i$, $\g_j$, and suppose that there is a unitary Lie algebra isomorphism between them. If they are both abelian, one of them is in the family $\g_0$ and the other in $\g_6$. If they are isomorphic to $\mathfrak{heis}_3\oplus\R^k$, their inner derivations are nilpotent of rank one; however, this only occurs for one Lie algebra in the family $\g_5(t)$, with $A(t)=0$. Otherwise, Freibert's criterion implies that the isomorphism preserves the codimension one ideal $\frh$, and $D,D'$ must be related by a transformation as in Lemma \ref{lemma:orbitsarbitrarydimension}. It is now clear that the only non-trivial possibilities are those listed in the statement.
\end{proof}

\begin{remark}
    Given $t\in\mathcal{U}_{p,q}$ and $c_1\in\R$, define the linear map $\phi\colon\frg_{4}(t,c_1)\to\frg_{5}(t)$ as follows: \begin{itemize}
        \item If $c_1\neq0$, $\phi(e_1)=\frac{1}{2c_1}(f_1-f_2)$, $\phi(e_2)=\frac{1}{2}(f_1+f_2)$ and $\phi(e_j)=f_j$ for $j=3,\ldots,2n$.
        \item If $c_1=0$, $\phi(e_1)=-f_2$, $\phi(e_2)=f_1$ and $\phi(e_j)=f_j$ for $j=3,\ldots,2n$.
    \end{itemize} In both cases $\phi\colon\frg_{4}(t,c_1)\to\frg_{5}(t)$ is an isomorphism of Lie algebras. However, the isomorphism cannot be chosen to be unitary since there is no transformation of the form \eqref{eqn:stabh} relating $D_4(t,c_1)$ and $D_5(t)$.
\end{remark}

\subsection{Classification in dimension six}\label{section:6d_classification}

Here we particularize the general classification obtained in Subsection \ref{section:classification_every_dim} to the six-dimensional case. We will consider separately the cases when $\frh_0$ is non-isotropic and when $\frh_0$ is isotropic.\medskip

Let us start with the non-isotropic case. Here we know that the almost abelian Lie algebra $\frg$ has to be isomorphic to $\g_0(t,a,\varepsilon)$, where $t\in\mathcal{U}_{p,q}$, with $p+q=2$, $a\in\R$ and $\varepsilon=\pm1$. The derivation $D$ takes the form $$D_0(t,a)=\begin{pmatrix}
    A(t)&0\\
    0&a
\end{pmatrix},$$ and $\frg_0(t,a,\varepsilon)$ is equipped with the following pseudo-Kähler structure: \begin{equation}\label{eqn:Jgnonisotropic6}
    Je_1=e_2,\,Je_3=e_4,\,Je_5=e_6\quad\text{and}\quad g=g(t)+\varepsilon(e^5\otimes e^5+e^6\otimes e^6),
\end{equation} where $g(t)$ is the real part of the Hermitian metric $h(t)$. Note that $(p,q)\in\{(2,0),(1,1),(0,2)\}$. The case $(0,2)$ means that we are considering a negative-definite metric on the subspace $\frh_1\subset\frg$, however we can always change the global sign of the metric and obtain a positive-definite one, so we will only consider the case $(2,0)$.\medskip

Let $\lambda_j\in\R$ and $\zeta_j=i\lambda_j$ for $j=1,2$. Then $$t=\Delta_0^+(\zeta_1)+\Delta_0^+(\zeta_2)=\left(\begin{pmatrix}
    \zeta_1&0\\
    0&\zeta_2
\end{pmatrix},\C\oplus\C,\begin{pmatrix}
    1&0\\
    0&1
\end{pmatrix}\right)$$ satisfies $s(t)=(2,0)$, where $s$ is the signature map \eqref{eq:signature_map}, and \begin{equation}\label{6d_noniso_(2,0)}
    D=\left(\begin{smallmatrix}
    0 & \lambda_1 & 0 & 0 & 0 \\
    -\lambda_1 & 0 & 0 & 0 & 0 \\
    0 & 0 & 0 & \lambda_2 & 0 \\
    0 & 0 & -\lambda_2 & 0 & 0 \\
    0 & 0 & 0 & 0 & a
\end{smallmatrix}\right),\quad g(t)=\begin{smallpmatrix}
    1 & 0 & 0 & 0 \\
    0 & 1 & 0 & 0 \\
    0 & 0 & 1 & 0 \\
    0 & 0 & 0 & 1
\end{smallpmatrix}.\end{equation}

Consider now the case $(p,q)=(1,1)$, which corresponds to taking a metric of neutral signature on $\frh_1$. Here we have different types $t\in\mathcal{U}$ such that $s(t)=(1,1)$: \begin{itemize}
    \item Let $\lambda\in\R$ and $\zeta=i\lambda$. Then $$t=\Delta_1^+(\zeta)=\left(\begin{pmatrix}
        \zeta&1\\
        0&\zeta
    \end{pmatrix},\C^2,\begin{pmatrix}
        0&i\\
        -i&0
    \end{pmatrix}\right)$$ and \begin{equation}\label{6d_noniso_(1,1)-I}
        D=\left(\begin{smallmatrix}
        0 & \lambda & 1 & 0 & 0 \\
        -\lambda & 0 & 0 & 1 & 0 \\
        0 & 0 & 0 & \lambda & 0 \\
        0 & 0 & -\lambda & 0 & 0 \\
        0 & 0 & 0 & 0 & a
    \end{smallmatrix}\right),\quad g(t)=\begin{smallpmatrix}
        0 & 0 & 0 & 1 \\
        0 & 0 & -1 & 0 \\
        0 & -1 & 0 & 0 \\
        1 & 0 & 0 & 0
    \end{smallpmatrix}.\end{equation}
    \item Let $\lambda_j\in\R$ and $\zeta_j=i\lambda_j$ for $j=1,2$. Then $$t=\Delta_0^+(\zeta_1)+\Delta_0^-(\zeta_2)=\left(\begin{pmatrix}
        \zeta_1&0\\
        0&\zeta_2
    \end{pmatrix},\C\oplus\C,\begin{pmatrix}
        1&0\\
        0&-1
    \end{pmatrix}\right)$$ and \begin{equation}\label{6d_noniso_(1,1)-II}
        D=\left(\begin{smallmatrix}
        0 & \lambda_1 & 0 & 0 & 0 \\
        -\lambda_1 & 0 & 0 & 0 & 0 \\
        0 & 0 & 0 & \lambda_2 & 0 \\
        0 & 0 & -\lambda_2 & 0 & 0 \\
        0 & 0 & 0 & 0 & a
    \end{smallmatrix}\right),\quad g(t)=\begin{smallpmatrix}
        1 & 0 & 0 & 0 \\
        0 & 1 & 0 & 0 \\
        0 & 0 & -1 & 0 \\
        0 & 0 & 0 & -1
    \end{smallpmatrix}.\end{equation}
    \item Let $\rho>0$, $\lambda\in\R$ and $\zeta=\rho+i\lambda$. Then $$t=\Delta_0(\zeta,-\overline{\zeta})=\left(\begin{pmatrix}
        \zeta&0\\
        0&-\overline{\zeta}
    \end{pmatrix},\C\oplus\C,\begin{pmatrix}
        0&i\\
        -i&0
    \end{pmatrix}\right)$$ and \begin{equation}\label{6d_noniso_(1,1)-III}
        D=\left(\begin{smallmatrix}
        \rho&\lambda&0&0&0\\
        -\lambda&\rho&0&0&0\\
        0&0&-\rho&\lambda&0\\
        0&0&-\lambda&-\rho&0\\
        0&0&0&0&a
    \end{smallmatrix}\right),\quad g(t)=\begin{smallpmatrix}
        0 & 0 & 0 & 1 \\
        0 & 0 & -1 & 0 \\
        0 & -1 & 0 & 0 \\
        1 & 0 & 0 & 0
    \end{smallpmatrix}.\end{equation}
\end{itemize}

\begin{proposition}\label{prop:class_6d_non-isotropic}
    Let $(\frg=\frh\rtimes_D{\Span{e_6}},J,g)$ be an almost abelian pseudo-Kähler Lie algebra of dimension six with $e_6$ non-isotropic. Then there is basis $e_1,\dotsc, e_5$ of $\frh$ such that $(J,g)$ takes the form \eqref{eqn:Jgnonisotropic6} and the derivation $D$ and the metric $g(t)$ are as in one of \eqref{6d_noniso_(2,0)}, \eqref{6d_noniso_(1,1)-I}, \eqref{6d_noniso_(1,1)-II} or \eqref{6d_noniso_(1,1)-III}.
\end{proposition}

We focus now on the isotropic case. Here we know that the almost abelian Lie algebra $\frg$ has to be isomorphic to $\g_i(t,\dots)$, for $i\in\{1,\ldots,6\}$, where $t\in\mathcal{U}_{p,q}$, with $p+q=1$. The derivation $D$ takes the form of $D_i(t,\dots)$ for $i\in\{1,\ldots,6\}$; and $\frg_i(t,\dots)$ is equipped with the following pseudo-Kähler structure: \begin{equation}\label{eqn:Jgisotropic6}
    Je_1=e_2,\,Je_3=e_4,\,Je_5=e_6\quad\text{and}\quad g=e^1\odot e^6-e^2\odot e^5+g(t),
\end{equation} where $g(t)$ is the real part of the Hermitian metric $h(t)$. The only possible type $t\in\mathcal{U}_{1,0}$ is $$t=\Delta_0^+(\zeta)=\big((\zeta),\C,(1)\big),$$ where $\zeta=i\lambda\in i\R$, thus $g(t)=\begin{smallpmatrix}
    1&0\\
    0&1
\end{smallpmatrix}$. Therefore the matrices $D_1,D_4,D_5,D_6$ take the form: \begin{equation}\label{6d_iso_D1456}
    \begin{smallpmatrix}
        -1 & 0 & 0 & 0 & 0 \\
        0 & -1 & 0 & 0 & c_{2} \\
        0 & 0 & 0 & \lambda & 0 \\
        0 & 0 & -\lambda & 0 & 0 \\
        0 & 0 & 0 & 0 & 1
    \end{smallpmatrix},\quad\begin{smallpmatrix}
        0 & 0 & 0 & 0 & c_{1} \\
        0 & 0 & 0 & 0 & 1 \\
        0 & 0 & 0 & \lambda & 0 \\
        0 & 0 & -\lambda & 0 & 0 \\
        0 & 0 & 0 & 0 & 0
    \end{smallpmatrix},\quad\begin{smallpmatrix}
        0 & 0 & 0 & 0 & 1 \\
        0 & 0 & 0 & 0 & 0 \\
        0 & 0 & 0 & \lambda & 0 \\
        0 & 0 & -\lambda & 0 & 0 \\
        0 & 0 & 0 & 0 & 0
    \end{smallpmatrix},\quad\begin{smallpmatrix}
        0 & 0 & 0 & 0 & 0 \\
        0 & 0 & 0 & 0 & 0 \\
        0 & 0 & 0 & \lambda & 0 \\
        0 & 0 & -\lambda & 0 & 0 \\
        0 & 0 & 0 & 0 & 0
    \end{smallpmatrix},
\end{equation} with $\lambda,c_1,c_2\in\R$. For the matrices $D_2$ and $D_3$ we need to find the vector $v(x)\in V$ as in \eqref{eqn:vofx}. Note that in these two cases, $t\in\mathcal{U}_{1,0}$ contains a nilpotent block, hence $t=\Delta_0^+(0)$, so $A=0$ and the set of nilpotent blocks in $t$ is just $$N_t=\{\Delta_0^+(0)\}.$$

Let $v=v(\Delta_0^+(0))$ be the last element in the standard basis of $V(\Delta_0^+(0))=\C$, so that $V(\Delta_0^+(0))=\Span{v}$. Since $\abs{h(v,v)}=1$, we can simply take $v=1\in\C$. Let $x=x(\Delta_0^+(0))\in\R_{\geq0}$. Then $$v(x)=\sum_{\Delta\in N_t}x(\Delta)v(\Delta)=x(\Delta_0^+(0))v(\Delta_0^+(0))=x\cdot1=\begin{pmatrix}
    x\\0
\end{pmatrix}\in\R^2.$$

Since in $D_2$ we can divide each entry by $x$, we get the following form of the matrices $D_2$ and $D_3$: \begin{equation}\label{6d_iso_D23}
    \left(\begin{smallmatrix}
        0 & 0 & 0 & -1 & 0 \\
        0 & 0 & 1 & 0 & 0 \\
        0 & 0 & 0 & 0 & 1 \\
        0 & 0 & 0 & 0 & 0 \\
        0 & 0 & 0 & 0 & 0
    \end{smallmatrix}\right),\quad \left(\begin{smallmatrix}
        0 & 0 & 0 & -x & 0 \\
        0 & 0 & x & 0 & 1 \\
        0 & 0 & 0 & 0 & x \\
        0 & 0 & 0 & 0 & 0 \\
        0 & 0 & 0 & 0 & 0
    \end{smallmatrix}\right).
\end{equation}

\begin{proposition}\label{prop:class_6d_isotropic}
    Let $(\frg=\frh\rtimes_D\Span{e_6},J,g)$ be an almost abelian pseudo-Kähler Lie algebra of dimension six with $e_6$ isotropic. Then there is basis $e_1,\dotsc, e_5$ of $\frh$ such that $J$ takes the form \eqref{eqn:Jgisotropic6}, $g=e^1\odot e^6-e^2\odot e^5+e^3\otimes e^3+e^4\otimes e^4$, and the derivation $D$ is one of \eqref{6d_iso_D1456} or one of \eqref{6d_iso_D23}.
\end{proposition}

\subsection{Classification in dimension eight}\label{section:8d_classification}

Here we use the results of Subsection \ref{section:classification_every_dim} to obtain a classification of eight-dimensional almost abelian pseudo-Kähler Lie algebras. We proceed as in Subsection \ref{section:6d_classification}.\medskip

Suppose first that $\frh_0$ is non-isotropic. The almost abelian Lie algebra $\frg$ has to be isomorphic to $\g_0(t,a,\varepsilon)$, where $t\in\mathcal{U}_{p,q}$, with $p+q=3$, $a\in\R$ and $\varepsilon=\pm1$. The derivation $D$ takes the form $$D_0(t,a)=\begin{pmatrix}
    A(t)&0\\
    0&a
\end{pmatrix},$$ and $\frg_0(t,a,\varepsilon)$ is equipped with the following pseudo-Kähler structure: \begin{equation}\label{eqn:Jgnonisotropic8}
    Je_1=e_2,\,Je_3=e_4,\,Je_5=e_6,\,Je_7=e_8\quad\text{and}\quad g=g(t)+\varepsilon(e^7\otimes e^7+e^8\otimes e^8),
\end{equation} where $g(t)$ is the real part of the Hermitian metric $h(t)$. For the case $(p,q)=(3,0)$, let $\lambda_j\in\R$ and $\zeta_j=i\lambda_j$ for $j=1,2,3$. Then $$t=\Delta_0^+(\zeta_1)+\Delta_0^+(\zeta_2)+\Delta_0^+(\zeta_3)=\left(\begin{pmatrix}
    \zeta_1&0&0\\
    0&\zeta_2&0\\
    0&0&\zeta_3
\end{pmatrix},\C\oplus\C\oplus\C,\begin{pmatrix}
    1&0&0\\
    0&1&0\\
    0&0&1
\end{pmatrix}\right)$$ and \begin{equation}\label{8d_noniso_(3,0)}
    D=\begin{smallpmatrix}
        0 & \lambda_{1} & 0 & 0 & 0 & 0 & 0 \\
        -\lambda_{1} & 0 & 0 & 0 & 0 & 0 & 0 \\
        0 & 0 & 0 & \lambda_{2} & 0 & 0 & 0 \\
        0 & 0 & -\lambda_{2} & 0 & 0 & 0 & 0 \\
        0 & 0 & 0 & 0 & 0 & \lambda_{3} & 0 \\
        0 & 0 & 0 & 0 & -\lambda_{3} & 0 & 0 \\
        0 & 0 & 0 & 0 & 0 & 0 & a
    \end{smallpmatrix},\quad g(t)=\begin{smallpmatrix}
        1 & 0 & 0 & 0 & 0 & 0 \\
        0 & 1 & 0 & 0 & 0 & 0 \\
        0 & 0 & 1 & 0 & 0 & 0 \\
        0 & 0 & 0 & 1 & 0 & 0 \\
        0 & 0 & 0 & 0 & 1 & 0 \\
        0 & 0 & 0 & 0 & 0 & 1
    \end{smallpmatrix}.
\end{equation}

Consider now the case $(p,q)=(2,1)$. Here, as in the six-dimensional case, we also have different types $t\in\mathcal{U}$ such that $s(t)=(2,1)$: \begin{itemize}
    \item Let $\lambda_j\in\R$ and $\zeta_j=i\lambda_j$ for $j=1,2$. Then $$t=\Delta_1^+(\zeta_1)+\Delta_0^+(\zeta_2)=\left(\begin{pmatrix}
        \zeta_1&1&0\\
        0&\zeta_1&0\\
        0&0&\zeta_2
    \end{pmatrix},\C^2\oplus\C,\begin{pmatrix}
        0&i&0\\
        -i&0&0\\
        0&0&1
    \end{pmatrix}\right)$$ and \begin{equation}\label{8d_noniso_(2,1)-I}
        D=\left(\begin{smallmatrix}
        0 & \lambda_{1} & 1 & 0 & 0 & 0 & 0 \\
        -\lambda_{1} & 0 & 0 & 1 & 0 & 0 & 0 \\
        0 & 0 & 0 & \lambda_{1} & 0 & 0 & 0 \\
        0 & 0 & -\lambda_{1} & 0 & 0 & 0 & 0 \\
        0 & 0 & 0 & 0 & 0 & \lambda_{2} & 0 \\
        0 & 0 & 0 & 0 & -\lambda_{2} & 0 & 0 \\
        0 & 0 & 0 & 0 & 0 & 0 & a
    \end{smallmatrix}\right),\quad g(t)=\begin{smallpmatrix}
        0 & 0 & 0 & 1 & 0 & 0 \\
        0 & 0 & -1 & 0 & 0 & 0 \\
        0 & -1 & 0 & 0 & 0 & 0 \\
        1 & 0 & 0 & 0 & 0 & 0 \\
        0 & 0 & 0 & 0 & 1 & 0 \\
        0 & 0 & 0 & 0 & 0 & 1
    \end{smallpmatrix}.\end{equation}
    \item Let $\lambda_j\in\R$ and $\zeta_j=i\lambda_j$ for $j=1,2,3$. Then $$t=\Delta_0^+(\zeta_1)+\Delta_0^+(\zeta_2)+\Delta_0^-(\zeta_3)=\left(\begin{pmatrix}
        \zeta_1&0&0\\
        0&\zeta_2&0\\
        0&0&\zeta_3
    \end{pmatrix},\C\oplus\C\oplus\C,\begin{pmatrix}
        1&0&0\\
        0&1&0\\
        0&0&-1
    \end{pmatrix}\right)$$ and \begin{equation}\label{8d_noniso_(2,1)-II}
        D=\left(\begin{smallmatrix}
        0 & \lambda_{1} & 0 & 0 & 0 & 0 & 0 \\
        -\lambda_{1} & 0 & 0 & 0 & 0 & 0 & 0 \\
        0 & 0 & 0 & \lambda_{2} & 0 & 0 & 0 \\
        0 & 0 & -\lambda_{2} & 0 & 0 & 0 & 0 \\
        0 & 0 & 0 & 0 & 0 & \lambda_{3} & 0 \\
        0 & 0 & 0 & 0 & -\lambda_{3} & 0 & 0 \\
        0 & 0 & 0 & 0 & 0 & 0 & a
    \end{smallmatrix}\right),\quad g(t)=\begin{smallpmatrix}
        1 & 0 & 0 & 0 & 0 & 0 \\
        0 & 1 & 0 & 0 & 0 & 0 \\
        0 & 0 & 1 & 0 & 0 & 0 \\
        0 & 0 & 0 & 1 & 0 & 0 \\
        0 & 0 & 0 & 0 & -1 & 0 \\
        0 & 0 & 0 & 0 & 0 & -1
    \end{smallpmatrix}.\end{equation}
    \item Let $\rho_1>0$, $\lambda_1\in\R$ and $\zeta_1=\rho_1+i\lambda_1$; let $\lambda_2\in\R$ and $\zeta_2=i\lambda_2$. Then $$t=\Delta_0(\zeta_1,-\overline{\zeta}_1)+\Delta_0^+(\zeta_2)=\left(\begin{pmatrix}
        \zeta_1&0&0\\
        0&-\overline{\zeta}_1&0\\
        0&0&\zeta_2
    \end{pmatrix},\C\oplus\C\oplus\C,\begin{pmatrix}
        0&i&0\\
        -i&0&0\\
        0&0&1
    \end{pmatrix}\right)$$ and \begin{equation}\label{8d_noniso_(2,1)-III}
        D=\left(\begin{smallmatrix}
        \rho_{1} & \lambda_{1} & 0 & 0 & 0 & 0 & 0 \\
        -\lambda_{1} & \rho_{1} & 0 & 0 & 0 & 0 & 0 \\
        0 & 0 & -\rho_{1} & \lambda_{1} & 0 & 0 & 0 \\
        0 & 0 & -\lambda_{1} & -\rho_{1} & 0 & 0 & 0 \\
        0 & 0 & 0 & 0 & 0 & \lambda_{2} & 0 \\
        0 & 0 & 0 & 0 & -\lambda_{2} & 0 & 0 \\
        0 & 0 & 0 & 0 & 0 & 0 & a
    \end{smallmatrix}\right),\quad g(t)=\begin{smallpmatrix}
        0 & 0 & 0 & 1 & 0 & 0 \\
        0 & 0 & -1 & 0 & 0 & 0 \\
        0 & -1 & 0 & 0 & 0 & 0 \\
        1 & 0 & 0 & 0 & 0 & 0 \\
        0 & 0 & 0 & 0 & 1 & 0 \\
        0 & 0 & 0 & 0 & 0 & 1
    \end{smallpmatrix}.\end{equation}
    \item Let $\lambda\in\R$ and $\zeta=i\lambda$. Then $$t=\Delta_2^-(\zeta)=\left(\begin{pmatrix}
        \zeta&1&0\\
        0&\zeta&1\\
        0&0&\zeta
    \end{pmatrix},\C^3,\begin{pmatrix}
        0&0&-1\\
        0&1&0\\
        -1&0&0
    \end{pmatrix}\right)$$ and \begin{equation}\label{8d_noniso_(2,1)-IV}
        D=\left(\begin{smallmatrix}
        0 & \lambda & 1 & 0 & 0 & 0 & 0 \\
        -\lambda & 0 & 0 & 1 & 0 & 0 & 0 \\
        0 & 0 & 0 & \lambda & 1 & 0 & 0 \\
        0 & 0 & -\lambda & 0 & 0 & 1 & 0 \\
        0 & 0 & 0 & 0 & 0 & \lambda & 0 \\
        0 & 0 & 0 & 0 & -\lambda & 0 & 0 \\
        0 & 0 & 0 & 0 & 0 & 0 & a
    \end{smallmatrix}\right),\quad g(t)=\begin{smallpmatrix}
        0 & 0 & 0 & 0 & -1 & 0 \\
        0 & 0 & 0 & 0 & 0 & -1 \\
        0 & 0 & 1 & 0 & 0 & 0 \\
        0 & 0 & 0 & 1 & 0 & 0 \\
        -1 & 0 & 0 & 0 & 0 & 0 \\
        0 & -1 & 0 & 0 & 0 & 0
    \end{smallpmatrix}.\end{equation}
\end{itemize}

\begin{proposition}
    Let $(\frg=\frh\rtimes_D{\Span{e_8}},J,g)$ be an almost abelian pseudo-Kähler Lie algebra of dimension eight with $e_8$ non-isotropic. Then there is basis $e_1,\dotsc, e_7$ of $\frh$ such that $(J,g)$ takes the form \eqref{eqn:Jgnonisotropic8} and the derivation $D$ and the metric $g(t)$ are as in one of \eqref{8d_noniso_(3,0)}, \eqref{8d_noniso_(2,1)-I}, \eqref{8d_noniso_(2,1)-II}, \eqref{8d_noniso_(2,1)-III} or \eqref{8d_noniso_(2,1)-IV}.
\end{proposition}

We focus now on the isotropic case. The almost abelian Lie algebra $\frg$ has to be isomorphic to $\g_i(t,\dots)$, for $i\in\{1,\ldots,6\}$, where $t\in\mathcal{U}_{p,q}$, with $p+q=2$. The derivation $D$ takes the form of $D_i(t,\dots)$ for $i\in\{1,\ldots,6\}$; and $\frg_i(t,\dots)$ is equipped with the following pseudo-Kähler structure: $$Je_1=e_2,\,Je_3=e_4,\,Je_5=e_6,\,Je_7=e_8\quad\text{and}\quad g=e^1\odot e^8-e^2\odot e^7+g(t),$$ where $g(t)$ is the real part of the Hermitian metric $h(t)$. Since $V(t)$ is four-dimensional we have two possibilities: the metric $g(t)$ on $V(t)$ is positive-definite, i.e.\ $t\in\mathcal{U}_{2,0}$; or the metric $g(t)$ on $V(t)$ is of neutral signature, i.e.\ $t\in\mathcal{U}_{1,1}$.\medskip

Suppose first that the metric $g(t)$ is positive-definite, so that $g$ has signature $(6,2)$. Let $\zeta_j\in i\R$ for $j=1,2$ and let $$t_0=\Delta_0^+(\zeta_1)+\Delta_0^+(\zeta_2)\in\mathcal{U}_{2,0};$$ as we have already seen, $g(t_0)$ is as in \eqref{6d_noniso_(2,0)}. Hence the possible derivations $D$ here are $D_1(t_0,c_2)$, $D_4(t_0,c_1)$, $D_5(t_0)$ and $D_6(t_0)$, where $c_1,c_2\in\R$. For the matrices $D_2(t_0,x)$ and $D_3(t_0,x)$ we need to find the vector $v(x)\in V$ as in \eqref{eqn:vofx}. Note that in these two cases, $t_0\in\mathcal{U}_{2,0}$ must contain a nilpotent block, hence (up to interchanging $\zeta_1$ and $\zeta_2$) $\zeta_1=0$ and $t_0=\Delta_0^+(0)+\Delta_0^+(\zeta_2)$. So, the set of nilpotent blocks in $t_0$ is just $$N_{t_0}=\{\Delta_0^+(0)\},$$ and nonzero maps $x\colon N_t\to \R_{\geq0}$ can be identified with positive numbers $x=x(\Delta_0^+(0))$, i.e.\ $$v(x)=x v(\Delta_0^+(0))=\begin{pmatrix}
    x\\
    0\\
    0\\
    0
\end{pmatrix}\in\R^4.$$

Note that in the matrix $D_2$ we can rescale $x$ and flip the sign of $A$, so that $x=1$ and $\zeta_2=i\lambda$ with $\lambda\geq0$. Then the matrix that we obtain is as follows: \begin{equation}\label{8d_iso_positive_D2}
D_2(t_0,1)=\left(\begin{smallmatrix}
    0 & 0 & 0 & -1 & 0 & 0 & 0 \\
    0 & 0 & 1 & 0 & 0 & 0 & 0 \\
    0 & 0 & 0 & 0 & 0 & 0 & 1 \\
    0 & 0 & 0 & 0 & 0 & 0 & 0 \\
    0 & 0 & 0 & 0 & 0 & \lambda & 0 \\
    0 & 0 & 0 & 0 & -\lambda & 0 & 0 \\
    0 & 0 & 0 & 0 & 0 & 0 & 0
\end{smallmatrix}\right),\quad\lambda\geq0.\end{equation}

For $D_3$, there is no such symmetry, and we obtain: \begin{equation}\label{8d_iso_positive_D3}
D_3(t_0,x)=\left(\begin{smallmatrix}
    0 & 0 & 0 & -x & 0 & 0 & 0 \\
    0 & 0 & x & 0 & 0 & 0 & 1 \\
    0 & 0 & 0 & 0 & 0 & 0 & x \\
    0 & 0 & 0 & 0 & 0 & 0 & 0 \\
    0 & 0 & 0 & 0 & 0 & \lambda & 0 \\
    0 & 0 & 0 & 0 & -\lambda & 0 & 0 \\
    0 & 0 & 0 & 0 & 0 & 0 & 0
\end{smallmatrix}\right),\quad x\geq0,\quad\lambda\in\R.\end{equation}

\begin{proposition}
    Let $(\frg,J,g)$ be an almost abelian pseudo-Kähler Lie algebra of dimension eight and signature $(6,2)$ with $\frh_0$ isotropic. Then one of the following cases holds: \begin{itemize}
        \item $\frg$ is isomorphic to $\frg_1(t_0,c_2)$, $\frg_4(t_0,c_1)$, $\frg_5(t_0)$ or $\frg_6(t_0)$;
        \item $\frg\cong\frg_2(t_0,x)$ with $D_2(t_0,1)$ as in \eqref{8d_iso_positive_D2};
        \item $\frg\cong\frg_3(t_0,x)$ with $D_3(t_0,x)$ as in \eqref{8d_iso_positive_D3}.
    \end{itemize}
\end{proposition}

Finally, suppose that the metric $g(t)$ is of neutral signature, so that $g$ is neutral too. As we have seen, there are three possibilities for $t\in\mathcal{U}_{1,1}$: \begin{itemize}
    \item $t_1=\Delta_1^+(\zeta)$ with $\zeta\in i\R$; and $g(t_1)$ as in \eqref{6d_noniso_(1,1)-I}.
    \item $t_2=\Delta_0^+(\zeta_1)+\Delta_0^-(\zeta_2)$ with $\zeta_1,\zeta_2\in i\R$; and $g(t_2)$ as in \eqref{6d_noniso_(1,1)-II}.
    \item $t_3=\Delta_0(\zeta,-\overline{\zeta})$ with $\zeta\in\C$ satisfying $\Re\zeta>0$; and $g(t_3)$ as in \eqref{6d_noniso_(1,1)-III}.
\end{itemize}

For each of these three cases, we can construct the corresponding matrices $D_1,D_4,D_5,D_6$. To construct the matrices $D_2$ and $D_3$ we need to consider the nilpotent blocks for each of the above cases. First note that $t_3$ has no nilpotent blocks, so $N_{t_3}=\varnothing$. Let us consider the other two cases: \begin{itemize}
    \item The set of nilpotent blocks of $t_1$ is $$N_{t_1}=\{\Delta_1^+(0)\}.$$

    Let $v=v(\Delta_1^+(0))$ be the last element in the standard basis of the vector space $V(\Delta_1^+(0))=\C^2$, so that $V(\Delta_1^+(0))=\Span{A(v),v}$. Since $\abs{h(A(v),v)}=1$, we can simply take $$v=\begin{pmatrix}
        0\\1
    \end{pmatrix}\in\C^2.$$

    Let $x=x(\Delta_1^+(0))\in\R_{\geq0}$. Then $$v(x)=\sum_{\Delta\in N_{t_1}}x(\Delta)v(\Delta)=x(\Delta_1^+(0))v(\Delta_1^+(0))=\begin{pmatrix}
        0\\0\\x\\0
    \end{pmatrix}\in\R^4.$$

    Therefore the matrices $D_2$ and $D_3$ are as follows: \begin{equation}\label{8d_iso_neutral_t1_D23}
    D_2(t_1,x)=\left(\begin{smallmatrix}
        0 & 0 & 0 & 0 & 0 & x & 0 \\
        0 & 0 & 0 & 0 & -x & 0 & 0 \\
        0 & 0 & 0 & 0 & 1 & 0 & 0 \\
        0 & 0 & 0 & 0 & 0 & 1 & 0 \\
        0 & 0 & 0 & 0 & 0 & 0 & x \\
        0 & 0 & 0 & 0 & 0 & 0 & 0 \\
        0 & 0 & 0 & 0 & 0 & 0 & 0
    \end{smallmatrix}\right),\quad D_3(t_1,x)=\left(\begin{smallmatrix}
        0 & 0 & 0 & 0 & 0 & x & 0 \\
        0 & 0 & 0 & 0 & -x & 0 & 1 \\
        0 & 0 & 0 & 0 & 1 & 0 & 0 \\
        0 & 0 & 0 & 0 & 0 & 1 & 0 \\
        0 & 0 & 0 & 0 & 0 & 0 & x \\
        0 & 0 & 0 & 0 & 0 & 0 & 0 \\
        0 & 0 & 0 & 0 & 0 & 0 & 0
    \end{smallmatrix}\right).\end{equation}
    \item The set of nilpotent blocks of $t_2$ is $$N_{t_2}=\{\Delta_0^+(0),\Delta_0^-(0)\}.$$

    Let $v^\pm=v(\Delta_0^\pm(0))$ be the last element in the standard basis of the vector space $V(\Delta_0^\pm(0))=\C$, so that $V(\Delta_0^\pm(0))=\Span{v^\pm}$. Since $\abs{h(v^\pm,v^\pm)}=1$, we just take $v^\pm=1\in\C$. Let $x^\pm=x(\Delta_0^\pm(0))\in\R_{\geq0}$. Then $$v(x)=\sum_{\Delta\in N_{t_2}}x(\Delta)v(\Delta)=x(\Delta_0^+(0))v(\Delta_0^+(0))+x(\Delta_0^-(0))v(\Delta_0^-(0))=\begin{pmatrix}
        x^+\\0\\x^-\\0
    \end{pmatrix}\in\R^4.$$

    Recall that the maps $x\colon N_t\to\R_{\geq0}$ satisfy the condition \eqref{eq:condition_Xt}. Therefore, if $x^-=0$, we can divide each entry of $D_2$ by $x^+$, so the matrices $D_2$ and $D_3$ are as follows: \begin{equation}\label{8d_iso_neutral_t2_D23_I}
    D_2(t_2,(1,0))=\left(\begin{smallmatrix}
        0 & 0 & 0 & -1 & 0 & 0 & 0 \\
        0 & 0 & 1 & 0 & 0 & 0 & 0 \\
        0 & 0 & 0 & 0 & 0 & 0 & 1 \\
        0 & 0 & 0 & 0 & 0 & 0 & 0 \\
        0 & 0 & 0 & 0 & 0 & 0 & 0 \\
        0 & 0 & 0 & 0 & 0 & 0 & 0 \\
        0 & 0 & 0 & 0 & 0 & 0 & 0
    \end{smallmatrix}\right),\quad D_3(t_2,(x^+,0))=\left(\begin{smallmatrix}
        0 & 0 & 0 & -x^+ & 0 & 0 & 0 \\
        0 & 0 & x^+ & 0 & 0 & 0 & 1 \\
        0 & 0 & 0 & 0 & 0 & 0 & x^+ \\
        0 & 0 & 0 & 0 & 0 & 0 & 0 \\
        0 & 0 & 0 & 0 & 0 & 0 & 0 \\
        0 & 0 & 0 & 0 & 0 & 0 & 0 \\
        0 & 0 & 0 & 0 & 0 & 0 & 0
    \end{smallmatrix}\right).\end{equation}

    If $x^+=0$, we can divide each entry of $D_2$ by $x^-$, so the matrices $D_2$ and $D_3$ are as follows: \begin{equation}\label{8d_iso_neutral_t2_D23_II}
    D_2(t_2,(0,1))=\left(\begin{smallmatrix}
        0 & 0 & 0 & 0 & 0 & 1 & 0 \\
        0 & 0 & 0 & 0 & -1 & 0 & 0 \\
        0 & 0 & 0 & 0 & 0 & 0 & 0 \\
        0 & 0 & 0 & 0 & 0 & 0 & 0 \\
        0 & 0 & 0 & 0 & 0 & 0 & 1 \\
        0 & 0 & 0 & 0 & 0 & 0 & 0 \\
        0 & 0 & 0 & 0 & 0 & 0 & 0
    \end{smallmatrix}\right),\quad D_3(t_2,(0,x^-))=\left(\begin{smallmatrix}
        0 & 0 & 0 & 0 & 0 & x^- & 0 \\
        0 & 0 & 0 & 0 & -x^- & 0 & 1 \\
        0 & 0 & 0 & 0 & 0 & 0 & 0 \\
        0 & 0 & 0 & 0 & 0 & 0 & 0 \\
        0 & 0 & 0 & 0 & 0 & 0 & x^- \\
        0 & 0 & 0 & 0 & 0 & 0 & 0 \\
        0 & 0 & 0 & 0 & 0 & 0 & 0
    \end{smallmatrix}\right).\end{equation}

    If $x^+=1=x^-$, the matrices $D_2$ and $D_3$ are as follows: \begin{equation}\label{8d_iso_neutral_t2_D23_III}
    D_2(t_2,(1,1))=\left(\begin{smallmatrix}
        0 & 0 & 0 & -1 & 0 & 1 & 0 \\
        0 & 0 & 1 & 0 & -1 & 0 & 0 \\
        0 & 0 & 0 & 0 & 0 & 0 & 1 \\
        0 & 0 & 0 & 0 & 0 & 0 & 0 \\
        0 & 0 & 0 & 0 & 0 & 0 & 1 \\
        0 & 0 & 0 & 0 & 0 & 0 & 0 \\
        0 & 0 & 0 & 0 & 0 & 0 & 0
    \end{smallmatrix}\right),\quad D_3(t_2,(1,1))=\left(\begin{smallmatrix}
        0 & 0 & 0 & -1 & 0 & 1 & 0 \\
        0 & 0 & 1 & 0 & -1 & 0 & 1 \\
        0 & 0 & 0 & 0 & 0 & 0 & 1 \\
        0 & 0 & 0 & 0 & 0 & 0 & 0 \\
        0 & 0 & 0 & 0 & 0 & 0 & 1 \\
        0 & 0 & 0 & 0 & 0 & 0 & 0 \\
        0 & 0 & 0 & 0 & 0 & 0 & 0
    \end{smallmatrix}\right).\end{equation}
\end{itemize}

\begin{proposition}
    Let $(\frg,J,g)$ be an almost abelian pseudo-Kähler Lie algebra of dimension eight and neutral signature with $\frh_0$ isotropic. Then one of the following cases holds: \begin{itemize}
        \item $\frg$ is isomorphic to $\frg_1(t_j,c_2)$, $\frg_4(t_j,c_1)$, $\frg_5(t_j)$ or $\frg_6(t_j)$ for $j=1,2,3$;
        \item $\frg\cong\frg_i(t_1,x)$, $i=1,2$, with $D_i(t_1,x)$ as in \eqref{8d_iso_neutral_t1_D23};
        \item $\frg\cong\frg_i(t_2,x)$, $i=1,2$, with $D_i(t_2,x)$ as in \eqref{8d_iso_neutral_t2_D23_I}, \eqref{8d_iso_neutral_t2_D23_II} or \eqref{8d_iso_neutral_t2_D23_III}.
    \end{itemize}
\end{proposition}


\section{Almost abelian Lie algebras admitting a pseudo-Kähler structure}\label{section:complex_PK}

In this section we consider pseudo-Kähler Lie algebras up to Lie algebra isomorphism, rather than unitary isomorphism, and classify the almost abelian Lie algebras that admit a pseudo-Kähler structure. We compare with the classification of \cite{ArroyoBarberisDiazGodoyHernandez} to obtain a classification of almost abelian Lie algebras that admit both a complex and a symplectic structure but no pseudo-Kähler structure.\medskip

We already know that $\R^{2n-3}\oplus\mathfrak{heis}_3$ has a pseudo-Kähler structure, so this problem reduces to classifying the conjugacy classes of the derivations $D$ that appear in the families $D_0,\dotsc, D_6$. In fact, the conjugacy classes of derivations $D$ such that $\frg=\R^{2n-1}\rtimes_D\R$ has a complex structure were determined in \cite{ArroyoBarberisDiazGodoyHernandez}. After recalling this result and restating it in our language, we will characterize the conjugacy classes of derivations $D$ that admit a pseudo-Kähler structure, and deduce that every nilpotent almost abelian Lie algebra admitting a complex structure also admits a pseudo-Kähler structure.\medskip

It will be convenient to introduce ad hoc notation analogous to the one of Subsection \ref{section:orbits}, whilst making contact with the notation of \cite{ArroyoBarberisDiazGodoyHernandez}. We introduce the semigroup parametrizing adjoint orbits of each $\gl(k,\R)$, $k\geq1$. This semigroup will be denoted by $\mathcal{GL}$; its generators are pairs $(V,X)$, with $V$ a real vector space and $X\colon V\to V$ an $\R$-linear map. Given $\alpha\in\R$, denote by $\mathcal{J}_{m+1}(\alpha)$ the real part of the complex matrix $A_{\alpha,m}$ of \eqref{eqn:jordanblock}, and set $$J_m(\alpha):=(\R^{m+1},\mathcal{J}_{m+1}(\alpha));$$ the shift in the index is due to conflicting notation between \cite{ArroyoBarberisDiazGodoyHernandez} and \cite{Burgoyne_Cushman_1977}. Given $\zeta\in\C\setminus\R$, let $\mathcal{C}_{m+1}(\zeta)$ be the real matrix corresponding to $A_{\zeta,m}$ under the isomorphism $\C^{m+1}\cong\R^{2m+2}$, and set $$C_m(\zeta):=\begin{cases}
    (\R^{2m+2},\mathcal{C}_{m+1}(\zeta))&\text{if }\im\zeta>0,\\
    (\R^{2m+2},\mathcal{C}_{m+1}(\overline\zeta))&\text{if }\im\zeta<0;
\end{cases}$$ since the complexification of $\mathcal{C}_{m+1}(\zeta)$ has eigenvalues $\zeta,\overline\zeta$ and $\mathcal{C}_{m+1}(\zeta)$, $\mathcal{C}_{m+1}(\overline\zeta)$ are in the same adjoint orbit, this definition ensures that $C_m(\zeta)$, $C_m(\eta)$ coincide if and only if they correspond to matrices in the same adjoint orbit. We define $\mathcal{GL}$ as the free abelian semigroup generated by the $J_m(\alpha)$ and $C_m(\zeta)$, with $\alpha\in\R$, $\zeta\in\C\setminus\R$, $m\geq0$.\medskip

The analogous of Theorem \ref{thm:orbits_U(p,q)} is the Jordan decomposition. To every element $X$ of $\gl(k,\R)$, we can associate the formal sum $j(X)$ of the $J_m(\alpha)$ and $C_m(\zeta)$ that appear in the Jordan decomposition of $X$, taken with multiplicity; $j(X)$ identifies the adjoint orbit in which $X$ lies, and will be called the \emph{Jordan type} of $X$. We will also consider the Jordan type $j(X)$ of matrices in $\gl(k,\C)$, which is also an element of $\mathcal{GL}$, obtained by the identification $\C^k\cong \R^{2k}$; for instance, if $X$ contains a single complex Jordan block with eigenvalue $\zeta$, $j(X)$ equals $2J_m(\zeta)$ or $C_m(\zeta)$ accordingly to whether $\zeta$ is real or not. The semigroup generated by $2J_m(\alpha)$ and $C_m(\zeta)$ will be denoted by $\mathcal{GL}_\C$; a matrix in $\gl(2k,\R)$ commutes with some complex structure on $\R^{2k}$ if and only if its Jordan type lies in $\mathcal{GL}_\C$.\medskip

Given matrices $M_1,\dots, M_k$, write $$M_1\oplus\cdots\oplus M_k:=\begin{pmatrix}
    M_1&\\
    &\ddots\\
    &&M_k
\end{pmatrix}.$$

In \cite{ArroyoBarberisDiazGodoyHernandez}, the generic nilpotent matrix is written as \begin{equation}\label{eqn:genericnilpotent}
    M=\mathcal{J}_{n_1}^{\oplus{p_1}}(0)\oplus\dots\oplus\mathcal{J}_{n_k}^{\oplus{p_k}}(0)\oplus 0_\tau, \quad n_1>\cdots >n_k\geq 2,\, p_i\geq 1,\, \tau\geq 0,
\end{equation} where $0_\tau$ denotes the zero square matrix of order $\tau$. In our language $$j(M)=p_1J_{n_1-1}(0) + \cdots + p_kJ_{n_k-1}(0) + \tau J_0(0).$$

\begin{theorem}[{\cite[Theorem 4.10]{ArroyoBarberisDiazGodoyHernandez}}]\label{thm:AA_complex}
    Let $\g=\R^{2n-1}\rtimes_D\R$ be a nilpotent almost abelian Lie algebra, with $D$ as in \eqref{eqn:genericnilpotent}. Then $\g$ admits a complex structure if and only if one of the following holds: \begin{enumerate}
        \item $\tau$ is odd and $p_1,\dotsc, p_k$ are even;
        \item $\tau$ is odd, $p_1,\dotsc, p_{k-1}$ are even, $n_k=2$ and $p_k$ is odd;
        \item $\tau$ is even, $k\geq 2$, there exists $2\leq\ell\leq k$ such that $n_{\ell-1}=n_\ell+1$, and all the $p_i$ are even except $p_{\ell-1}$, $p_\ell$, which are odd.
    \end{enumerate}
\end{theorem}

In the same paper, the authors show:

\begin{theorem}[{\cite[Theorem 4.17]{ArroyoBarberisDiazGodoyHernandez}}]\label{thm:AA_complex_general}
    Let $\g=\R^{2n-1}\rtimes_D\R$ be an almost abelian Lie algebra. Then $\g$ admits a complex structure if and only if $D$ has a real eigenvalue $\alpha$ such that $$D=M(\alpha)\quad\text{or}\quad D=M(\alpha)\oplus Q,$$ where $M(\alpha)-\alpha\Id\in\End(\R^{2n-1-2m})$ satisfies the conditions of Theorem \ref{thm:AA_complex} and $Q\in\End(\R^{2m})$ commutes with a complex structure on $\R^{2m}$ and does not have $\alpha$ as an eigenvalue.
\end{theorem}

In our language, we can summarize Theorems \ref{thm:AA_complex} and \ref{thm:AA_complex_general} as follows:

\begin{corollary}\label{cor:ofArroyoetal}
    An almost abelian Lie algebra $\frg=\R^{2n-1}\rtimes_D\R$ admits a complex structure if and only if one of the following holds: \begin{itemize}
        \item $j(D)\in J_0(\alpha)+\mathcal{GL}_\C$ with $\alpha\in\R$; or
        \item $j(D)\in J_k(\alpha)+J_{k+1}(\alpha)+\mathcal{GL}_\C$ with $\alpha\in\R$ and $k\geq0$.
    \end{itemize}
\end{corollary}

\begin{proof}
    Let $\mathcal{T}$ denote the semigroup generated by $2J_m(\alpha)$ as $\alpha\in\R$, $m\geq0$. Consider the case where $D$ is nilpotent. With this assumption, Theorem \ref{thm:AA_complex} shows that $\g=\R^{2n-1}\rtimes_D\R$ has a complex structure if and only if the Jordan type belongs to one of the following $$J_0(0)+\mathcal{T},\quad J_0(0)+J_1(0)+\mathcal{T},\quad J_k(0)+J_{k+1}(0)+\mathcal{T},\quad k\geq1,$$ though only the eigenvalue zero is allowed to occur; notice that the second case is subsumed in the last if $k\geq0$ is allowed.\medskip

    Now assume $\R^{2n-1}\rtimes_D\R$ has a complex structure, with $D$ arbitrary. By Theorem \ref{thm:AA_complex_general} we can assume $D=M(\alpha)\oplus Q$, where $M(\alpha)$ has type $$J_0(\alpha)+\mathcal{T},\quad  J_k(\alpha)+J_{k+1}(\alpha)+\mathcal{T},\quad k\geq0,$$ and $Q$ commutes with a complex structure, hence $j(Q)\in\mathcal{GL}_\C$. Since $\mathcal{T}\subset\mathcal{GL}_\C$, we obtain $$j(D)=j(M(\alpha))+j(Q)\in J_0(\alpha)+\mathcal{GL}_\C\quad\text{or}\quad j(D)\in J_k(\alpha)+J_{k+1}(\alpha)+\mathcal{GL}_\C,\quad k\geq0.$$

    Conversely, if $D$ has Jordan type in $J_0(\alpha)+\mathcal{GL}_\C$, write $$j(D)=J_0(\alpha)+r+q,$$ where $r$ is a sum of elements of the form $J_m(\alpha)$ and $q$ does not contain any $J_m(\alpha)$. We see that $J_0(\alpha)+r$ is in $J_0(\alpha)+\mathcal{T}$, and $q$ is in $\mathcal{GL}_\C$. Let $M(\alpha)$ be a matrix with Jordan type $J_0(\alpha)+r$. If $q=0$, up to conjugation we can assume $D=M(\alpha)$; otherwise, we can assume $D=M(\alpha)\oplus Q$, with $j(Q)=q$. In both cases, $\R^{2n-1}\rtimes_D\R$ has a complex structure by Theorem \ref{thm:AA_complex_general}. The same argument applies when $D$ has type in $J_k(\alpha)+J_{k+1}(\alpha)+\mathcal{GL}_\C$.
\end{proof}

Denote by $\mathcal{Q}$ the semigroup generated by $j(A(t))$ as $t\in\mathcal{U}$. Explicitly, the generators are $$2J_m(0),\quad   2J_m(x)+2J_m(-x)\text{ with }x\in\R\setminus\{0\},$$ $$C_m(\zeta)\text{ with }\zeta\in i\R\setminus\{0\},\quad C_m(\zeta)+C_m(-\overline\zeta)\text{ with }\zeta\in\C\setminus(\R\cup i\R).$$

\begin{theorem}\label{thm:jordantypes}
    Let $\g=\R^{2n-1}\rtimes_D\R$ be an almost abelian Lie algebra. Then $\g$ admits a pseudo-Kähler structure if and only if there exists $\alpha\in\R$ such that $j(D)$ is one of the following cases: $$J_0(\alpha)+\mathcal{Q}, \quad J_m(0)+J_{m+1}(0)+\mathcal{Q}, \quad J_0(\alpha)+2J_0(-\alpha)+\mathcal{Q}.$$
\end{theorem}

\begin{proof}
    Suppose $\R^{2n-1}\rtimes_D\R$ is an abelian pseudo-Kähler Lie algebra. Then a nonzero multiple of $D$ is conjugated to a derivation in the families $D_0,\dotsc, D_6$; we need to compute $j(D)$ in each case. The following cases are straightforward: \begin{itemize}
        \item $j(D_0(t,a))=j(A(t))+J_0(a)$,
        \item $j(D_1(t,c_2))=j(A(t))+J_0(1)+2J_0(-1)$,
        \item $j(D_4(t,c_1))=j(A(t))+J_1(0)+J_0(0)=j(D_5(t))$,
        \item $j(D_6(t))=j(A(t))+3J_0(0)$.
    \end{itemize}

    It is clear that all these Jordan types are of the form in the statement, and all cases are covered  except $J_m(0)+J_{m+1}(0)+\mathcal{Q}$ for $m>0$.\medskip

    For $D=D_2(t,x)$, let $m$ be the largest integer such that some $\Delta=\Delta_m^\varepsilon(0)\in N_t$ satisfies $x(\Delta)>0$, meaning that $v(x)$ has a nonzero component along $v(\Delta)$. By \eqref{eq:condition_Xt}, $\Delta$ is unique unless both $\Delta_m^+(0)$ and $\Delta_m^-(0)$ appear in $t$ and $x(\Delta_m^+(0))=1=x(\Delta_m^-(0))$. Because of the form of the metric, we have $$h_m(v(\Delta),A^kv(\Delta))=\begin{cases}
        1&\text{if }k=m\text{ is even},\\
        -i&\text{if }k=m\text{ is odd},\\
        0&\text{otherwise}.
    \end{cases}$$

    In real terms, this implies \begin{equation}\label{eqn:gvAkv}
        g(\Delta)(v(\Delta),A^kv(\Delta))=\begin{cases}
            \varepsilon&\text{if }k=m\text{ is even},\\
            0&\text{otherwise},
        \end{cases}\qquad g(\Delta)(Jv(\Delta),A^kv(\Delta))=\begin{cases}
            \varepsilon&\text{if }k=m\text{ is odd},\\
            0&\text{otherwise}.
        \end{cases}
    \end{equation}

    Consider the case where $x(\Delta_m^+(0))=1=x(\Delta_m^-(0))$. It follows from \eqref{eqn:gvAkv} that $A^m(v(x))$ is orthogonal to $v(x)$ and $Jv(x)$; hence, $e_{2n-1}$ generates a Jordan block of type $J_{m+1}(0)$, i.e.\ $\Span{D^k(e_{2n-1})\mid k\geq 0}$ has dimension $m+2$. Similarly, $v(\Delta_{m}^+(0))$, $Jv(\Delta_{m}^+(0))$ generate Jordan blocks of type $J_{m+1}(0)$, and $Jv(x)$ generates a Jordan block of type $J_m(0)$. The other Jordan blocks of $A$ can be assumed to be in the kernel of $v(x)^\flat$ and $(Jv(x))^\flat$ by a (non-unitary) change of basis, so they determine Jordan blocks of $D$ of the same type. Thus, in this case $$j(D_2(t,x))=(j(A)-4J_m(0))+J_{m}(0)+3J_{m+1}(0)\in J_{m}(0)+J_{m+1}(0)+\mathcal{Q},\quad m\geq0.$$

    Suppose now that $x(\Delta^{\pm}_m(0)$ is only nonzero for one choice of the sign $\pm1$, i.e.\ the block $\Delta$ chosen above is unique. Then $$\Span{D^n(e_{2n-1})\mid n\geq0}=\Span{e_{2n-1},v(x), Av(x),\dotsc, A^{m}v(x)}\oplus\begin{cases}
        \Span{e_1}&\text{if }m\text{ odd},\\
        \Span{e_2}&\text{if }m\text{ even},
    \end{cases}$$ i.e.\ $e_{2n-1}$ generates a Jordan block of type $J_{m+2}(0)$; similarly, $Jv(x)$ generates a block of type $J_{m+1}(0)$, containing $e_1$ or $e_2$ according to whether $m$ is even or odd. Again, we can perform a non-unitary change of basis that puts the other blocks in the kernel of $v(x)^\flat$ and $(Jv(x))^\flat$, and we obtain $$j(D_2(t,x))=j(A(t))-2J_m(0)+J_{m+1}(0)+J_{m+2}(0)\in J_{m+1}(0)+J_{m+2}(0)+\mathcal{Q}, \quad m\geq0.$$

    Thus, the Jordan types corresponding to $D_2(t,x)$ belong to the second family in the statement and cover the missing cases. Notice that for $D_3(t,x)$ we obtain the same Jordan types, since $D_3(t,x)$ and $D_2(t,x)$ are conjugated in $\gl(2n,\R)$.
\end{proof}

\begin{corollary}\label{cor:complex_PK}
    Let $\g$ be a nilpotent almost abelian Lie algebra admitting a complex structure. Then $\g$ admits a pseudo-Kähler structure.
\end{corollary}

\begin{proof}
    If $D$ is nilpotent, all of its blocks are of the form $J_m(0)$; writing $D$ as in Corollary \ref{cor:ofArroyoetal}, the component of $j(D)$ in $\mathcal{GL}_\C$ also belongs to $\mathcal{Q}$; it now suffices to apply Theorem \ref{thm:jordantypes}.
\end{proof}

More generally, an almost abelian Lie algebra that admits a complex structure does not necessarily admit a symplectic structure. Thus, Corollary \ref{cor:complex_PK} does not generalize in a straightforward way to the non-nilpotent case. A more subtle question is whether every almost abelian Lie algebra that admits both a complex and a symplectic structure also admits a pseudo-Kähler structure. In the last part of this section, we will answer this question in the negative.\medskip

The starting point is the classification of almost abelian Lie algebras admitting a symplectic form obtained in \cite{ArroyoBarberisDiazGodoyHernandez}. The classification is broken up in three results; we will state the first in the same form as the original, summarize the second, and reformulate the third in our language.

\begin{theorem}[{\cite[Theorem 5.10]{ArroyoBarberisDiazGodoyHernandez}}]\label{thm:symplecticnilpotent}
    A nilpotent almost abelian Lie algebra $\frg=\R^{2n-1}\rtimes_D\R$ admits a symplectic structure if and only if $D$ takes the form \eqref{eqn:genericnilpotent} with either \begin{enumerate}
        \item $\tau$ odd and $p_i$ even for all odd $n_i$;
        \item $\tau$ even and there exists a unique $\ell$ such that both $n_\ell$ and $p_\ell$ are odd.
    \end{enumerate}
\end{theorem}

\begin{theorem}[{\cite[Theorem 5.24]{ArroyoBarberisDiazGodoyHernandez}}]\label{thm:symplectic_c}
    Let $c\in\R\setminus\{0\}$; fix a matrix $M(c)\in\End(\R^n)$ which only has the eigenvalue $c$, and $M(-c)\in\End(\R^{n-1})$ which only has the eigenvalue $-c$. Set $D=M(c)\oplus M(-c)$. Then $\frg=\R^{2n-1}\rtimes_D\R$ admits a symplectic form if and only if up to conjugation either $$M(c)=-M(-c)\oplus\mathcal{J}_1(c),$$ or there exists a matrix $A\in\End(\R^{n-k-1})$, $k\geq1$, which only has the eigenvalue $c$ such that $$M(c)=A\oplus\mathcal{J}_{k+1}(c),\quad M(-c)=-A\oplus\mathcal{J}_{k}(-c).$$
\end{theorem}

If $\omega\in\Lambda^2(\R^{2n})^*$ is non-degenerate, we define $$\lie{sp}_\omega=\{A\in\End(\R^{2n})\mid \omega(Av,w)+\omega(v,Aw)=0,\ \ v,w\in\R^{2n}\}.$$

By \cite[Theorem 2.4]{Mehl_Rodman_2016}, a matrix in $\gl(k,\R)$ lies in $\lie{sp}_\omega$ for some non-degenerate $\omega$ if and only if its Jordan type lies in the semigroup spanned by $$2J_{2m}(0),\quad J_{2m+1}(0),\quad J_m(x)+J_m(-x)\text{ with }x\in\R\setminus\{0\},$$ $$C_m(\zeta)+C_m(-\zeta)\text{ with }\zeta\in\C\setminus(\R\cup i\R), \quad C_m(\zeta)\text{ with }\zeta\in i\R,$$ which we will denote by $\mathcal{GL}_{\lie{sp}}$.

\begin{corollary}\label{cor:symplectic}
    An almost abelian Lie algebra $\frg=\R^{2n-1}\rtimes_D\R$ admits a symplectic form if and only if $j(D)$ belongs to one of the following: $$J_{2m}(0)+\mathcal{GL}_{\lie{sp}},\quad J_{0}(c)+\mathcal{GL}_{\lie{sp}},\quad J_m(-c)+J_{m+1}(c)+\mathcal{GL}_{\lie{sp}}, \quad c\in\R\setminus\{0\}.$$
\end{corollary}

\begin{proof}
    By \cite[Theorem 5.26]{ArroyoBarberisDiazGodoyHernandez}, $D$ defines a symplectic almost abelian Lie algebra if and only if, up to conjugation, it falls in one of these categories: \begin{enumerate}
        \item $D=M(0)\oplus Q$, where $Q$ is in $\lie{sp}_\omega$ for some $\omega$, $M(0)$ is as in Theorem \ref{thm:symplecticnilpotent}, and $Q$ does not have $0$ as an eigenvalue. This implies $$j(D)\in J_{2m}(0)+\mathcal{GL}_{\lie{sp}}, \quad m\geq0.$$
        \item $D=M(c)\oplus M(-c)\oplus Q$, where $M(c)$, $M(-c)$ are as in Theorem \ref{thm:symplectic_c}, and $Q$ is a matrix in $\lie{sp}_\omega$ for some $\omega$ which does not have $\pm c$ as an eigenvalue. This gives $$j(D)\in J_{0}(c)+\mathcal{GL}_{\lie{sp}}\quad\text{or}\quad j(D)\in J_m(-c)+J_{m+1}(c)+\mathcal{GL}_{\lie{sp}},\quad m\geq0.$$
        \item $D=\mathcal{J}_1(c)\oplus Q$, where $Q$ is as in the second point. This results in $$j(D)\in J_0(c)+\mathcal{GL}_{\lie{sp}}.$$
    \end{enumerate}
    It is clear that every matrix with Jordan type as in the statement can be written in one of the above forms.
\end{proof}

\begin{theorem}\label{thm:notpk}
    The almost abelian Lie algebra $\frg=\R^{2n-1}\rtimes_D\R$ admits both a complex and a symplectic structure but not a pseudo-Kähler structure if and only if $$j(D)\in 2J_m(-c)+J_m(c)+J_{m+1}(c)+\mathcal{Q}, \quad c\in\R\setminus\{0\},\quad m\geq0.$$
\end{theorem}

\begin{proof}
    We first characterize derivations $D$ such that $\R^{2n-1}\rtimes_D\R$ admits both a complex and a symplectic structure. For such a $D$, write $$j(D)=j'+j_\C=j''+j_{\lie{sp}},\quad j_\C\in\mathcal{GL}_\C,\quad j_{\lie{sp}}\in \mathcal{GL}_{\lie{sp}},$$ where by Corollary \ref{cor:ofArroyoetal} and Corollary \ref{cor:symplectic} we have \begin{equation}\label{eqn:jprimejsecond}
        \begin{aligned}
            j'&\in\{J_0(\alpha), J_k(\alpha) + J_{k+1}(\alpha)\mid k\geq0,\alpha\in\R\},\\
            j''&\in\{J_{2m}(0), J_0(c), J_m(-c)+J_{m+1}(c)\mid m\geq0,c\neq0\}.
        \end{aligned}
    \end{equation}
    
    In order to relate $j'$ and $j''$, we replace the free abelian semigroups $\mathcal{GL}$ with the free abelian group over the same generators, denoted $\GGL$. We will use analogous notation for the other semigroups we have defined. By construction $$j''-j'=j_\C-j_{\lie{sp}}\in\GGL_\C+\GGL_{\lie{sp}}.$$
    
    Working mod $\overline{\mathcal{Q}}$, we have \begin{equation}\label{eqn:jjprimejsecond}
        \frac{\GGL_\C+\GGL_{\lie{sp}}}{\overline{\mathcal{Q}}}=\frac{\GGL_\C}{\overline{\mathcal{Q}}}\oplus \frac{\GGL_{\lie{sp}}}{\overline{\mathcal{Q}}},
    \end{equation} where the first factor is generated by cosets of elements $2J_m(x)$ and $C_m(\zeta)$, and the second factor by cosets of $J_{2m+1}(0)$, $J_m(x)+J_m(-x)$; as usual, $x$ varies in $\R\setminus\{0\}$ and $\zeta$ in $\C\setminus(\R\cup i\R)$. It follows from \eqref{eqn:jjprimejsecond} that given $j'$ and $j''$, $j_\C$ and $j_{\lie {sp}}$ are uniquely determined mod $\overline{\mathcal{Q}}$. Considering the possible choices in \eqref{eqn:jprimejsecond}, we obtain (writing $j_\C,j_{\lie{sp}}$ mod $\mathcal{Q}$): \begin{enumerate}
        \item $j'=J_0(0)=j''$, $j_\C=0=j_{\lie{sp}}$;
        \item $j'=J_{2m}(0)+J_{2m+1}(0)$, $j''=J_{2m}(0)$, $j_\C=0$, $j_{\lie{sp}}= J_{2m+1}(0)$;
        \item $j'=J_{2m+1}(0)+J_{2m+2}(0)$, $j''=J_{2m+2}(0)$, $j_\C=0$, $j_{\lie{sp}}= J_{2m+1}(0)$;
        \item $j'=J_0(c)=j''$, $j_\C=0=j_{\lie{sp}}$;
        \item $j'=J_0(c)$, $j''=J_0(-c)$, $j_\C=2J_0(-c)$, $j_{\lie{sp}}=J_0(c)+J_0(-c)$;
        \item $j'=J_m(c)+J_{m+1}(c)$, $j''=J_m(-c)+J_{m+1}(c)$, $j_\C=2J_m(-c)$, $j_{\lie{sp}}=J_m(c)+J_m(-c)$.
    \end{enumerate}
    
    Summing up, if $\R^{2n-1}\rtimes_D\R$ admits both a complex and a symplectic structure, $j(D)$ is obtained by adding an element of $\mathcal{Q}$ to an element that has one of the following forms: $$J_0(0),\quad J_{2m}(0)+J_{2m+1}(0),\quad J_{2m+1}(0)+J_{2m+2}(0),$$ $$J_0(c),\quad J_0(c)+2J_0(-c),\quad 2J_m(-c)+J_m(c)+J_{m+1}(c).$$
    
    Conversely, it is clear that any $D$ obtained in this way determines a Lie algebra that admits both a complex and a symplectic structure. Comparing with Theorem \ref{thm:jordantypes}, we see that those that do not admit a pseudo-Kähler structure are precisely those in the last family.
\end{proof}

\begin{example}\label{example:6d_CS_noPK}
    The blocks $2J_m(-c)+J_m(c)+J_{m+1}(c)$ in Theorem~\ref{thm:notpk} can only appear in a matrix of order at least $4m+5$. Thus, the lowest possible dimension of an almost abelian Lie algebra that can admit both a complex and a symplectic structure but not a pseudo-K\"ahler structure is six; in this case, the Lie algebra is determined by the derivation $$D=\begin{smallpmatrix}
        -c & 0 & 0 & 0 & 0\\
        0 & -c & 0 & 0 & 0\\
        0 & 0 & c & 0 & 0\\
        0 & 0 & 0 & c & 1\\
        0 & 0 & 0 & 0 & c
    \end{smallpmatrix};$$ the parameter $c\neq0$ can be eliminated by rescaling.
\end{example}


\section{Pseudo-Kähler-Einstein extensions}\label{section:PKE_extensions}

In this final section we apply the results of \cite{Conti_Rossi_SegnanDalmasso_2023} to obtain a pseudo-Kähler-Einstein Lie algebra of nonzero scalar curvature $(\tilde{\frg},\tilde{J},\tilde{g})$ from a nilpotent almost abelian pseudo-Kähler Lie algebra $(\frg,J,g)$. This construction boils down to finding a particular derivation $\check{D}$ of $\frg$. We write down the general algebraic equations that determine the admissible $\check D$ for a fixed $(\frg,J,g)$, and exemplify the construction by obtaining all the pseudo-Kähler-Einstein extensions (in this sense) of a six-dimensional non-abelian nilpotent almost abelian pseudo-Kähler Lie algebra. We will not consider the case where $\g$ is abelian, which is completely described in \cite{Conti_Rossi_SegnanDalmasso_2023}.\medskip

First of all we recall the result from \cite{Conti_Rossi_SegnanDalmasso_2023}, adjusting a sign for our convention $\omega(X,Y)=g(JX,Y)$. As a matter of notation, we will write $\g\oplus_{\alpha}\Span{b}$ for the central extension obtained from $\g$ by declaring that $\d b^*=\alpha$, with $\{b^*\}$ the dual basis of $\{b\}$.

\begin{proposition}[{\cite[Corollary 3.6]{Conti_Rossi_SegnanDalmasso_2023}}]\label{prop:constructse}
    Let $(\frg,J,g)$ be a nilpotent pseudo-Kähler Lie algebra and let $\check{D}$ be a derivation of $\frg$ of the form $\check{D}=\Id+\check{D}^a$, where $\check{D}^a$ is skew-symmetric relative to $g$ and commutes with $J$. Then the semidirect product $\tilde{\frg}=(\frg\oplus_{-2\omega}\Span{b})\rtimes\Span{e_0}$, where $$[e_0,X]=\check{D}X,\quad[e_0,b]=2b,\quad X\in\frg,$$ has a pseudo-Kähler structure $(\tilde{J},\tilde{g})$ given by $$\tilde{g}=g-b^*\otimes b^*-e^0\otimes e^0,\quad \tilde{J}(X)=JX,\quad \tilde{J}(b)=e_0,\quad \tilde{J}(e_0)=-b,\quad X\in\frg,$$ satisfying $\widetilde{\Ric}=(\dim\tilde{\g}+2)\tilde{g}$.
\end{proposition}

We particularize Proposition \ref{prop:constructse} to the almost abelian case. For that we need the general form of a derivation on an almost abelian Lie algebra.

\begin{lemma}\label{lemma:derivationsalmostabelian}
    Let $\frg=\frh\rtimes_D\Span{e_{2n}}$ be an almost abelian Lie algebra. For $K\in\End(\frh)$, $\alpha\in\frh^*$, $u\in\frh$, $k\in\R$, the linear endomorphism of $\frg$ given by $$K+\alpha\otimes e_{2n}+e^{2n}\otimes (ke_{2n}+u)$$ is a derivation if and only if $[K,D]=kD$, and either $\alpha=0$ or $\Im D\subset\ker\alpha\subset\ker D$.
\end{lemma}

\begin{proof}
    By definition, the linear map is a derivation if and only if $$0=[\alpha(X)e_{2n},Y]+[X,\alpha(Y)e_{2n}]=\alpha(X)DY-\alpha(Y)DX,$$
    $$KDX+\alpha(DX)e_{2n}=[e_{2n},KX]+[u+ke_{2n},X]=DKX+kDX,$$ for all $X,Y\in\frh$. The first equation gives that $\alpha(X)Y-\alpha(Y)X$ lies in $\ker D$; as $X,Y$ vary, this means that either $\alpha=0$ or $\ker\alpha\subset\ker D$. The second equation implies that $\alpha\circ D=0$ and $[K,D]=kD$.
\end{proof}

\begin{remark}
    As observed in \cite{Avetisyan_2022}, if $\alpha$ is nonzero then $\ker D\supset \ker\alpha\supset\im D$ has at most codimension one, i.e.\ $D$ is zero or nilpotent of rank one. This only occurs when $\g$ is abelian or the direct sum of an abelian factor with the three-dimensional Heisenberg Lie algebra $\mathfrak{heis}_3$.
\end{remark}

We first consider the non-isotropic case.

\begin{proposition}\label{prop:ESnil_non-isotropic}
    Let $(\frg,J,g)$ be a nilpotent almost abelian pseudo-Kähler Lie algebra with $\frh_0$ non-isotropic, and let $\check D$ be a derivation of $\frg$ satisfying the conditions of Proposition \ref{prop:constructse}. Then \begin{equation}\label{eq:D_and_check_D_non_isotropic}
        D=\begin{pmatrix}
        A&0\\
        0&0
    \end{pmatrix},\quad\check{D}=\begin{pmatrix}
        K_1 & w & Jw\\
        -w^\flat & 1 & h\\
        -(Jw)^\flat & -h & 1
    \end{pmatrix},\end{equation} where $A$ is nilpotent, $K_1-\Id_{\frh_1}$ is skew-Hermitian, $[K_1,A]=A$, and $Aw=0$.
\end{proposition}

\begin{proof}
    Let $\check{D}$ be a derivation commuting with $J$ and such that $\check D^s=\Id$. In terms of the decomposition $\g_0(t,a,\varepsilon)=\h_1\oplus\Span{e_{2n-1}}\oplus\Span{e_{2n}}$, we can write $$\check{D}=\begin{pmatrix}
        K_1 & w & Jw\\
        -w^\flat & 1 & h\\
        -(Jw)^\flat & -h & 1
    \end{pmatrix},$$ with $K_1\in\End(\h_1)$, $w\in\h_1$, $h\in\R$; and where $K_1-\Id$ is skew-Hermitian. With the notation of Lemma \ref{lemma:derivationsalmostabelian}, we have $$K=K_1-w^\flat\otimes e_{2n-1} + e^{2n-1}\otimes (w + e_{2n-1}),$$ $$\alpha=-(Jw)^\flat - he^{2n-1},\quad u=Jw+he_{2n-1},\quad k=1.$$
    
    Writing $D$ as in \eqref{eq:D_positive_g0}, with $D$ nilpotent, and imposing $D=[K,D]$, we find $$\begin{pmatrix}
        A & 0 \\
        0 & 0
    \end{pmatrix}=\left[\begin{pmatrix}
        K_1 & w \\
        -w^\flat & 1
    \end{pmatrix},\begin{pmatrix}
        A & 0 \\
        0 & 0
    \end{pmatrix}\right]=\begin{pmatrix}
        [K_1,A] & -Aw \\
        (Aw)^\flat & 0
    \end{pmatrix}.$$

    Then $[K_1,A]=A$ and $Aw=0$.
\end{proof}

Let $(\frg,J,g)$ be a nilpotent almost abelian pseudo-Kähler Lie algebra of dimension six with $\frh_0$ non-isotropic. Then, by Proposition \ref{prop:class_6d_non-isotropic}, there is a basis $\{e_1,\ldots,e_6\}$ of $\frg$ such that $Je_1=e_2$, $Je_3=e_4$, $Je_5=e_6$, and the derivation $D=\ad(e_6)|_\frh$ and metric $g$ take the form \begin{equation}\label{eq:6d_nilpotent_non-isotropic}
D=\begin{smallpmatrix}
    0 & 0 & 1 & 0 & 0 \\
    0 & 0 & 0 & 1 & 0 \\
    0 & 0 & 0 & 0 & 0 \\
    0 & 0 & 0 & 0 & 0 \\
    0 & 0 & 0 & 0 & 0
\end{smallpmatrix},\quad g=e^1\odot e^4-e^2\odot e^3+\varepsilon(e^5\otimes e^5+e^6\otimes e^6)\end{equation} for $\varepsilon=\pm1$. Using this we obtain the following result.

\begin{proposition}\label{prop:KEextension6d_non_isotropic}
    The pseudo-Kähler-Einstein Lie algebras obtained by applying the construction of Proposition \ref{prop:constructse} to a non-abelian nilpotent almost abelian pseudo-Kähler Lie algebra of dimension six with $\frh_0$ non-isotropic are \begin{equation}\label{eqn:explicitKEext_noniso}
        \bigl((\R^5\rtimes_D\R)\oplus_{-2\omega}\Span{e_7}\bigr)\rtimes_{E}\Span{e_0},\quad \omega=-e^{13}-e^{24}+\varepsilon e^{56},
    \end{equation} where the pseudo-Kähler structure is given by \begin{gather*}
        \tilde{J}e_1=e_2,\,\tilde{J}e_3=e_4,\,\tilde{J}e_5=e_6,\,\tilde{J}e_7=e_0,\\
        \tilde{g}=e^1\odot e^4-e^2\odot e^3+\varepsilon(e^5\otimes e^5+e^6\otimes e^6)-e^7\otimes e^7-e^0\otimes e^0,
    \end{gather*} and for some real parameters $h,x,w_1,w_2\in\R$ we have $$E=\begin{smallpmatrix}
        \frac{3}{2} & x & 0 & 0 & w_1 & -w_2 & 0\\
        -x & \frac{3}{2} & 0 & 0 & w_2 & w_1 & 0\\
        0 & 0 & \frac{1}{2} & x & 0 & 0 & 0\\
        0 & 0 & -x & \frac{1}{2} & 0 & 0 & 0\\
        0 & 0 & w_2 & -w_1 & 1 & h & 0\\
        0 & 0 & w_1 & w_2 & -h & 1 & 0\\
        0 & 0 & 0 & 0 & 0 & 0 & 2\\
    \end{smallpmatrix}.$$
\end{proposition}

\begin{proof}
    Recalling the form of the metric and complex structure, we see that $\omega=-e^{13}-e^{24}+\varepsilon e^{56}$; setting $e_7=b$ and $E=\check{D}+2b^*\otimes b$, we obtain that any pseudo-Kähler-Einstein extension of a pseudo-Kähler almost abelian Lie algebra with $\frh_0$ non-isotropic has the form \eqref{eqn:explicitKEext_noniso}, where $D$ is as in \eqref{eq:6d_nilpotent_non-isotropic}, and $\check{D}$ satisfies $\check{D}=\Id+\check{D}^a$, with $\check{D}^a$ skew-Hermitian. Applying Proposition \ref{prop:ESnil_non-isotropic} we obtain that $$\check{D}=\begin{smallpmatrix}
        \frac{3}{2} & x & y & 0 & w_1 & -w_2\\
        -x & \frac{3}{2} & 0 & y & w_2 & w_1\\
        0 & 0 & \frac{1}{2} & x & 0 & 0\\
        0 & 0 & -x & \frac{1}{2} & 0 & 0\\
        0 & 0 & w_2 & -w_1 & 1 & h\\
        0 & 0 & w_1 & w_2 & -h & 1
    \end{smallpmatrix},$$ where $h,x,y,w_1,w_2\in\R$. The parameter $y$ can be set to zero; indeed, in the notation of Proposition \ref{prop:ESnil_non-isotropic}, the skew-Hermitian matrix $K_1-\Id$ admits a unitary basis of eigenvectors, and this change of frame does not affect $A$, whose image and kernel coincide with the eigenspace of $\frac32+ix$.
\end{proof}

We consider now the isotropic case.

\begin{proposition}\label{prop:ESnil}
    Let $(\frg,J,g)$ be a nilpotent almost abelian pseudo-Kähler Lie algebra with $\frh_0$ isotropic, and let $\check D$ be a derivation of $\frg$ satisfying the conditions of Proposition \ref{prop:constructse}. Then \begin{equation}\label{eq:D_and_check_D}
        D=\begin{pmatrix}
        0&0&-(J_Vv)^{\flat_V}&c_1\\
        0&0&v^{\flat_V}&c_2\\
        0&0&A&v\\
        0&0&0&0
    \end{pmatrix},\quad\check{D}=\begin{pmatrix}
        1+d_{11} & d_{12} & -(J_Vz)^{\flat_V} & d_{1,2n-1} & 0\\
        -d_{12} &1+ d_{11} & z^{\flat_V} &0 & d_{1,2n-1}\\
        w & J_Vw & P & z & J_Vz\\
        d_{2n-1,1} & 0 & (J_Vw)^{\flat_V} & 1-d_{11} & d_{12}\\
        0 & d_{2n-1,1} &- w^{\flat_V} & -d_{12} & 1-d_{11}
    \end{pmatrix},\end{equation} where $A$ is nilpotent, $P-\Id_V$ is skew-Hermitian, $[P,A]=0$ and $$Aw=0, \quad g(v,w)=0=g(v,Jw), \quad d_{2n-1,1}c_1=0=d_{2n-1,1}c_2, \quad d_{2n-1,1}v=0, \quad c_1w=0=c_2w,$$
    $$2d_{11}c_1 + d_{12}c_2 +2g(Jv,z)=0, \quad d_{12}v=0, \quad d_{12}c_1-2d_{11}c_2=0, \quad Az-Pv+(1-d_{11})v=0.$$
\end{proposition}

\begin{proof}
    Fix a basis $\{e_1,\ldots,e_{2n}\}$ of $\frg$ such that the almost pseudo-Hermitian structure $(J,g)$ takes the form \eqref{eqn:almosthermitianstructurezero}. Relative to the basis $\{e_1,\dotsc, e_{2n}\}$, the fact that $\check D-\Id$ is skew-Hermitian implies that $$\check{D}=\begin{pmatrix}
        1+d_{11} & d_{12} & -(J_Vz)^{\flat_V} & d_{1,2n-1} & 0\\
        -d_{12} &1+ d_{11} & z^{\flat_V} &0 & d_{1,2n-1}\\
        w & J_Vw & P & z & J_Vz\\
        d_{2n-1,1} & 0 & (J_Vw)^{\flat_V} & 1-d_{11} & d_{12}\\
        0 & d_{2n-1,1} &- w^{\flat_V} & -d_{12} & 1-d_{11}
    \end{pmatrix},$$ where $P-\Id_V$ is skew-Hermitian. Relative to the same basis, $D$ takes the form \eqref{eq:D_zero_g0}; since we are assuming $\g$ is nilpotent, we obtain $$D=\begin{pmatrix}
        0&0&-(J_Vv)^{\flat_V}&c_1\\
        0&0&v^{\flat_V}&c_2\\
        0&0&A&v\\
        0&0&0&0
    \end{pmatrix},$$ where $A$ is nilpotent. In the language of Lemma \ref{lemma:derivationsalmostabelian}, $K$ is the submatrix obtained from $\check D$ by eliminating the last column and row. By the same lemma, we have $[K,D]=kD$. Concentrating on the first two columns, we compute $$[K,D]=\begin{pmatrix}
        g(Jv,w)-c_1d_{2n-1,1}  & g(v,w) &* &*\\
        g(v,w)-c_2d_{2n-1,1} & -g(v,Jw) & * &*\\
        -Aw-d_{2n-1,1}v & -AJw & *&* \\
        0 & 0 &*&*
    \end{pmatrix}.$$

    Recalling that $A$ commutes with $J$, we obtain \begin{equation}\label{eqn:ESnil:twocolumns}
        Aw=0,\quad  g(v,w)=0=g(v,Jw), \quad d_{2n-1,1}c_1=0=d_{2n-1,1}c_2, \quad d_{2n-1,1}v=0.
    \end{equation}

    These conditions also imply that, in the language of Lemma \ref{lemma:derivationsalmostabelian}, $\alpha\circ D=0$. Making use of \eqref{eqn:ESnil:twocolumns}, we obtain $$[K,D]=\left(\begin{smallmatrix}
        0 & 0 & -(1+d_{11})(Jv)^{\flat_V} +d_{12}v^{\flat_V} -(J z)^{\flat_V} A+(J v)^{\flat_V} P -c_1 (Jw)^{\flat_V} & (1+d_{11})c_1 + d_{12}c_2 - g(Jz,v)+g(Jv,z)-c_1(1-d_{11})\\
        0 & 0 & d_{12}(Jv)^{\flat_V} +(1+d_{11})v^{\flat_V} +z^{\flat_V} A-v^{\flat_V} P-c_2 (Jw)^{\flat_V} & -d_{12}c_1 + (1+d_{11})c_2 + g(z,v)-g(z,v)-c_2(1-d_{11})\\
        0 & 0 & [P,A] & c_1w+c_2Jw+Pv-Az-(1-d_{11})v\\
        0 & 0 & 0 & 0
    \end{smallmatrix}\right).$$

    In particular, $[P,A]=kP$ is skew-Hermitian, which implies $k=0$ because $P-\Id_V$ is skew-Hermitian. Imposing $[K,D]=0$ we get \begin{align*}
        0&=-(1+d_{11})(Jv)^{\flat_V} +d_{12}v^{\flat_V} -(J z)^{\flat_V} A+(J v)^{\flat_V} P -c_1 (Jw)^{\flat_V},\\
        0&=(1+d_{11})c_1 + d_{12}c_2 - g(Jz,v)+g(Jv,z)-c_1(1-d_{11}),\\
        0&=d_{12}(Jv)^{\flat_V} +(1+d_{11})v^{\flat_V} +z^{\flat_V} A-v^{\flat_V} P-c_2 (Jw)^{\flat_V},\\
        0&=-d_{12}c_1 + 2d_{11}c_2,\\
        0&=c_1w+c_2Jw+Pv-Az-(1-d_{11})v.
    \end{align*}

    Raising indices, and recalling that $P^*=-P+2\Id$ we get \begin{align*}
        0&=-c_1 (Jw) +d_{12}v + A(J z)- P(J v)+(1-d_{11})(Jv),\\
        0&=2d_{11}c_1 + d_{12}c_2 +2g(Jv,z),\\
        0&=d_{12}(Jv)-c_2 (Jw) -(1-d_{11})v - Az+ Pv,\\
        0&=-d_{12}c_1 + 2d_{11}c_2,\\
        0&=c_1w+c_2Jw-Az+Pv-(1-d_{11})v.
    \end{align*}

    Taking $J$ of the first equation above we get $$c_1 w +d_{12}Jv - A z+Pv-(1-d_{11})v=0,$$ and subtracting the third equation we obtain $c_1w+c_2Jw=0$, which implies that $c_1w=0$ and $c_2w=0$. Thus, we are left with $$2d_{11}c_1 + d_{12}c_2 +2g(Jv,z)=0, \quad d_{12}v=0, \quad d_{12}c_1-2d_{11}c_2=0, \quad Az-Pv+(1-d_{11})v=0,$$ which completes the set of conditions from the statement.
\end{proof}

Let $(\frg,J,g)$ be a nilpotent almost abelian pseudo-Kähler Lie algebra of dimension six with $\frh_0$ isotropic. Then, by Proposition \ref{prop:class_6d_isotropic}, the nilpotent derivation $D$ takes one of the following forms: \begin{equation}\label{eq:6d_nilpotent_isotropic}
    D_2=\left(\begin{smallmatrix}
    0 & 0 & 0 & -1 & 0 \\
    0 & 0 & 1 & 0 & 0 \\
    0 & 0 & 0 & 0 & 1 \\
    0 & 0 & 0 & 0 & 0 \\
    0 & 0 & 0 & 0 & 0
\end{smallmatrix}\right),\quad D_3(x)=\left(\begin{smallmatrix}
    0 & 0 & 0 & -x & 0 \\
    0 & 0 & x & 0 & 1 \\
    0 & 0 & 0 & 0 & x \\
    0 & 0 & 0 & 0 & 0 \\
    0 & 0 & 0 & 0 & 0
\end{smallmatrix}\right),\quad D_4(c_1)=\left(\begin{smallmatrix}
    0 & 0 & 0 & 0 & c_{1} \\
    0 & 0 & 0 & 0 & 1 \\
    0 & 0 & 0 & 0 & 0 \\
    0 & 0 & 0 & 0 & 0 \\
    0 & 0 & 0 & 0 & 0
\end{smallmatrix}\right),\quad D_5=\left(\begin{smallmatrix}
    0 & 0 & 0 & 0 & 1 \\
    0 & 0 & 0 & 0 & 0 \\
    0 & 0 & 0 & 0 & 0 \\
    0 & 0 & 0 & 0 & 0 \\
    0 & 0 & 0 & 0 & 0
\end{smallmatrix}\right),\end{equation} where $x>0$ and $c_1\in\R$. Using these derivations we obtain the following result.

\begin{proposition}\label{prop:KEextension6d}
    The pseudo-Kähler-Einstein Lie algebras obtained by applying the construction of Proposition \ref{prop:constructse} to a non-abelian nilpotent almost abelian pseudo-Kähler Lie algebra of dimension six with $\frh_0$ isotropic are \begin{equation}\label{eqn:explicitKEext}
        \bigl((\R^5\rtimes_D\R)\oplus_{-2\omega}\Span{e_7}\bigr)\rtimes_{E}\Span{e_0},\quad \omega=-e^{15}-e^{26}+e^{34},
    \end{equation} where the pseudo-Kähler structure is given by \begin{gather*}
        \tilde{J}e_1=e_2,\,\tilde{J}e_3=e_4,\,\tilde{J}e_5=e_6,\,\tilde{J}e_7=e_0,\\
        \tilde{g}=e^1\odot e^6-e^2\odot e^5+e^3\otimes e^3+e^4\otimes e^4-e^7\otimes e^7-e^0\otimes e^0,
    \end{gather*} and for some real parameters $d_{15},z_1,z_2,p$ one of the following holds: \begin{itemize}
        \item $D$ equals either $D_2$ or $D_3(x)$, $x>0$, and $$E=\begin{smallpmatrix}
            1 & 0 & 0 & -z_{1} & d_{15} & 0 & 0\\
            0 & 1 & z_{1} & 0 & 0 & d_{15} & 0\\
            0 & 0 & 1 & 0 & z_{1} & 0 & 0\\
            0 & 0 & 0 & 1 & 0 & z_{1} & 0\\
            0 & 0 & 0 & 0 & 1 & 0 & 0\\
            0 & 0 & 0 & 0 & 0 & 1 & 0\\
            0 & 0 & 0 & 0 & 0 & 0 & 2\\
        \end{smallpmatrix};$$
        \item $D$ equals either $D_4(c_1)$, $c_1\in\R$, or $D_5$ and $$E=\begin{smallpmatrix}
            1 & 0 & z_{2} & -z_{1} & d_{15} & 0 & 0\\
            0 & 1 & z_{1} & z_{2} & 0 & d_{15} & 0\\
            0 & 0 & 1 & p & z_{1} & -z_{2} & 0\\
            0 & 0 & -p & 1 & z_{2} & z_{1} & 0\\
            0 & 0 & 0 & 0 & 1 & 0 & 0\\
            0 & 0 & 0 & 0 & 0 & 1 & 0\\
            0 & 0 & 0 & 0 & 0 & 0 & 2\\
        \end{smallpmatrix}.$$
    \end{itemize}
\end{proposition}

\begin{proof}
    Recalling the form of the metric and complex structure, we see that $\omega=-e^{15}-e^{26}+e^{34}$; setting $e_7=b$ and $E=\check{D}+2b^*\otimes b$, we obtain that any pseudo-Kähler-Einstein extension of a pseudo-Kähler almost abelian Lie algebra with $\frh_0$ isotropic has the form \eqref{eqn:explicitKEext}, where $D$ is one of \eqref{eq:6d_nilpotent_isotropic}, and $\check{D}$ satisfies $\check{D}=\Id+\check{D}^a$, with $\check{D}^a$ skew-Hermitian. Applying Proposition \ref{prop:ESnil}, we obtain the following general form: $$\check{D}=\begin{smallpmatrix}
        1+d_{11}&d_{12}&z_2&-z_1&d_{15}&0\\
        -d_{12}&1+d_{11}&z_1&z_2&0&d_{15}\\
        w_1&-w_2&1&p&z_1&-z_2\\
        w_2&w_1&-p&1&z_2&z_1\\
        d_{51}&0&-w_2&w_1&1-d_{11}&d_{12}\\
        0&d_{51}&-w_1&-w_2&-d_{12}&1-d_{11}
    \end{smallpmatrix},$$ where $d_{11},d_{12},d_{15},d_{51},z_1,z_2,w_1,w_2,p\in\R$, and its specific form depends on the derivation $D$. Using $D_2$ or $D_3(x)$ we obtain the first matrix from the statement, and using $D_4(c_1)$ or $D_5$ we obtain the second matrix from the statement. The same argument also shows that every Lie algebra of this form is pseudo-Kähler-Einstein.
\end{proof}


\bibliographystyle{myamsplain}
\bibliography{biblio}

\small\textsc{Dipartimento di Matematica, Università di Pisa, Largo Bruno Pontecorvo, 5, 56127 Pisa PI, Italy}\\
\textit{Email address:} \texttt{diego.conti@unipi.it}

\medskip

\small\textsc{Scuola Internazionale Superiore di Studi Avanzati (SISSA), via Bonomea, 265, 34136 Trieste TS, Italy}\\
\textit{Email address:} \texttt{agilgarc@sissa.it}

\end{document}

%% file: Preamble.tex
\usepackage[top=1in,bottom=1in,left=1in,right=1in]{geometry}
\usepackage{graphicx}
\usepackage{caption}
\usepackage{subcaption}
\usepackage{float} 
\usepackage{lmodern}
\usepackage[T1]{fontenc}
\usepackage[english,activeacute]{babel}
\usepackage{mathtools}
\usepackage{amsmath,amsthm,amssymb,amsfonts}
\usepackage{mathrsfs}
\usepackage{stmaryrd} 
\usepackage[noadjust]{cite} 
\usepackage{verbatim} 
\usepackage[utf8]{inputenc} 
\usepackage{soul} 
\usepackage{tikz-cd} 
\usetikzlibrary{babel} 
\usepackage{setspace} 
\usepackage{array,multicol,multirow,makecell}
\usepackage{subfiles}
\usepackage[nottoc]{tocbibind} 
\usepackage{longtable}
\usepackage{tikz}
\usetikzlibrary{positioning}
\usepackage{dynkin-diagrams}
\usepackage{authblk}
\usepackage{abstract}
\usepackage{bbm}

\usepackage[pdftex,backref=page]{hyperref} 

\hypersetup{linktocpage=true,colorlinks=true,citecolor=blue,urlcolor=blue}
\newcommand{\doi}[1]{\url{https://doi.org/#1}}

\usepackage{orcidlink}

\usepackage{import}
\usepackage{xifthen}
\usepackage{pdfpages}
\usepackage{transparent}

\let\OLDthebibliography\thebibliography
\renewcommand\thebibliography[1]{
	\OLDthebibliography{#1}
	\setlength{\parskip}{0pt}
}

\newtheorem{theorem}{Theorem}[section]
\newtheorem{corollary}[theorem]{Corollary}
\newtheorem{proposition}[theorem]{Proposition}
\newtheorem{lemma}[theorem]{Lemma}

\theoremstyle{definition}

\newtheorem{example}[theorem]{Example}
\newtheorem{remark}[theorem]{Remark}

\DeclareMathOperator{\Id}{Id}

\DeclareMathOperator{\End}{End}

\DeclareMathOperator{\ad}{ad}
\DeclareMathOperator{\Der}{Der}

\DeclareMathOperator{\rank}{rank}

\DeclareMathOperator{\Ric}{Ric}
\DeclareMathOperator{\ric}{ric}

\newcommand{\N}{\mathbb{N}}

\newcommand{\R}{\mathbb{R}}
\newcommand{\C}{\mathbb{C}}

\newcommand{\frg}{\mathfrak{g}}
\newcommand{\frh}{\mathfrak{h}}

\renewcommand{\Re}{\operatorname{Re}}
\renewcommand{\Im}{\operatorname{Im}}

\newcommand{\tran}[1]{\hspace{.2mm}\prescript{t\hspace{-.5mm}}{}{#1}}

\renewcommand{\d}{\mathrm{d}}

\DeclarePairedDelimiter{\abs}{\lvert}{\rvert}

\allowdisplaybreaks

\usepackage{paralist}

\newcommand{\Span}[1]{\operatorname{Span}\{#1\}}
\newcommand{\SpanC}[1]{\operatorname{Span}_\C\{#1\}}
\newcommand{\lie}[1]{\mathfrak{#1}}

\DeclareMathOperator{\Tr}{tr}
\DeclareMathOperator{\im}{Im}
\newcommand{\g}{\lie g}
\newcommand{\h}{\lie h}
\newcommand{\gl}{\lie{gl}}
\newenvironment{smallpmatrix}{\left(\begin{smallmatrix}}{\end{smallmatrix}\right)}